\documentclass[12pt]{article}
\usepackage{amsmath,amsfonts}
\allowdisplaybreaks[4]
\usepackage{verbatim}
\usepackage{latexsym}
\usepackage{graphicx}
\usepackage{float}
\usepackage{subfigure}
\usepackage{color,xcolor}
\usepackage{slashed}
\usepackage[bookmarksnumbered,bookmarksnumbered,bookmarksopen,colorlinks,citecolor=blue,linkcolor=blue]{hyperref}
\usepackage{bm}

\usepackage{indentfirst}

\makeatletter
\@addtoreset{equation}{section}

\begin{document}

\newcommand{\DD}{\mathbb{D}}
\newcommand{\EE}{\mathbb{E}}
\newcommand{\HH}{\mathbb{H}}
\newcommand{\LL}{\mathbb{L}}
\newcommand{\NN}{\mathbb{N}}
\newcommand{\PP}{\mathbb{P}}
\newcommand{\RR}{\mathbb{R}}
\newcommand{\SM}{\mathbb{S}}

\newcommand{\BB}{\mathcal{B}}
\newcommand{\CC}{\mathcal{C}}
\newcommand{\CF}{\mathcal{F}}
\newcommand{\CK}{\mathcal{K}}
\newcommand{\CJ}{\mathcal{J}}
\newcommand{\CL}{\mathcal{L}}
\newcommand{\CM}{\mathcal{M}}
\newcommand{\CP}{\mathcal{P}}
\newcommand{\CR}{\mathcal{R}}
\newcommand{\CU}{\mathcal{U}}
\newcommand{\CW}{\mathcal{W}}
\newcommand{\CX}{\mathcal{X}}
\newcommand{\CY}{\mathcal{Y}}
\newcommand{\II}{\mathbf{1}}
\newcommand{\hX}{\hat X}
\newcommand{\bq}{\bar q}
\newcommand{\bV}{\bar V}

\newtheorem{theorem}{Theorem}[section]
\newtheorem{lemma}[theorem]{Lemma}
\newtheorem{coro}[theorem]{Corollary}
\newtheorem{defn}[theorem]{Definition}
\newtheorem{assp}[theorem]{Assumption}
\newtheorem{expl}[theorem]{Example}
\newtheorem{prop}[theorem]{Proposition}
\newtheorem{rmk}[theorem]{Remark}

\newcommand\tq{{\scriptstyle{3\over 4 }\scriptstyle}}
\newcommand\qua{{\scriptstyle{1\over 4 }\scriptstyle}}
\newcommand\hf{{\textstyle{1\over 2 }\displaystyle}}
\newcommand\hhf{{\scriptstyle{1\over 2 }\scriptstyle}}

\newcommand{\proof}{\noindent {\it Proof}. }
\newcommand{\eproof}{\hfill $\Box$} 

\def\a{\alpha}      \def\al{\aleph}       \def\be{\beta}    \def\c{\check}       \def\d{\mathrm{d}}   \def\de{\delta}   \def\e{\varepsilon}  \def\ep{\epsilon}
\def\eq{\equiv}     \def\f{\varphi}   \def\g{\gamma}       \def\h{\forall}   \def\i{\bot}
\def\j{\emptyset}   \def\k{\kappa}    \def\lan{\langle}    \def\ran{\rangle} \def\lbd{\lambda}
\def\m{\mu}         \def\n{\nu}
\def\nn{\nonumber}  \def\o{\theta}    \def\p{\phi}         \def\q{\surd}     \def\r{\rho}
\def\ra{\rightarrow}
\def\s{\sigma}      \def\td{\tilde}   \def\up{\upsilon}
\def\vk{\varkappa}   \def\vr{\varrho}
\def\ve{\vee}     \def\vo{\vartheta}
\def\w{\omega}      \def\we{\wedge}     \def\x{\xi}          \def\y{\eta}      \def\z{\zeta}

\def\D{\Delta}     \def\F{\Phi}         \def\G{\Gamma}    \def\K{\times}
\def\L{\Lambda}    \def\M{\partial}     \def\N{\nabla}       \def\O{\Theta}    \def\S{\Sigma}
\def\T{\tau}       \def\U{\bigwedge}    \def\V{\bigvee}      \def\W{\Omega}    \def\X{\Xi}       \def\Ex{\exists}

\def\1{\oslash}   \def\2{\oplus}    \def\3{\otimes}      \def\4{\ominus}
\def\5{\circ}     \def\6{\odot}     \def\7{\backslash}   \def\8{\infty}
\def\9{\bigcap}   \def\0{\bigcup}   \def\+{\pm}          \def\-{\mp}
\def\la{\langle}  \def\ra{\rangle}  \def\ra{\rightarrow}

\def\tl{\tilde}
\def\trace{\hbox{\rm trace}}
\def\diag{\hbox{\rm diag}}
\def\for{\quad\hbox{for }}
\def\refer{\hangindent=0.3in\hangafter=1}

\newcommand\wD{\widehat{\D}}

\title{
\bf
Development of numerical methods for nonlinear hybrid stochastic functional differential equations with infinite delay
 }

\author{
{\bf
Guozhen Li ${}^{1}$,
Xiaoyue Li ${}^{2}$\thanks{Research
of this author  was supported by the National Natural Science Foundation of China (No. 12371402) and the Tianjin Natural Science Foundation (24JCZDJC00830).},
Xuerong Mao${}^{3}$\thanks{ Research of this author was supported by
the Royal Society (No. WM160014, Royal Society Wolfson Research Merit Award),
the Royal Society of Edinburgh (No. RSE1832).  }}
\\
${}^1$ School of Mathematics and Statistics, \\
Northeast Normal University, Changchun, Jilin, 130024, China. \\
${}^2$ School of Mathematical Sciences, \\
Tiangong University, Tianjin, 300387, China. \\
${}^3$ Department of Mathematics and Statistics, \\
University of Strathclyde, Glasgow G1 1XH, U.K. \\
}

\date{}

\maketitle

\begin{abstract}

This paper addresses the challenging numerical simulation of nonlinear hybrid stochastic functional differential equations with infinite delays. We first propose an explicit scheme using space and time truncation, requiring only finite historical storage. Leveraging approximation theory, we prove the boundedness of the numerical solution's $ p$th moment and establish its convergence, achieving a rate of $1/2$ order under polynomially growing coefficients. Furthermore, we refine the scheme to better capture the underlying exponential stability of the exact solution, in both moment and almost sure senses. Finally, numerical experiments are presented to validate our theoretical results.

\end{abstract}

\section{Introduction}
Long-term memory exists in many physical phenomena, such as viscoelasticity, population dynamics, thermodynamics. The research can be traced back to Boltzmann and Volterra \cite{B74, V12} and has made great progress \cite{C76, FGP10, K93, RHN87}.
The hybrid systems modulated by Markov chains involving continuous dynamics and discrete events are often used for modelling the systems experiencing the abrupt changes in wireless communications, signal processing, queueing networks \cite{MY06, YZ10}.
Hybrid stochastic functional differential equations (HSFDEs) with long-term memory attracted much attention owing to the significance as well as wide applications
\cite{LX21,SWW22, WYM17}.
This paper aims to study the HSFDE  with infinite delay  (HSFDEwID)
\begin{equation}\label{ISFDE}
\d x(t)= f(x_t, \o(t))\d t + g(x_t, \o(t))\d B(t),\quad t\ge 0
\end{equation}
with the initial data
\begin{equation}\label{2.3}
(x_0 , \o(0)) = (\xi, i_0) \in \CC_r\times \SM.
\end{equation}
Here, $x_t(u) = x(t+u),~ u \le 0,~ x_t(\cdot) \in \CC_r$,
$
f : \CC_r \K \SM \ra \RR^n,~g : \CC_r \K \SM \ra \RR^{n \K d},
$
$\theta(\cdot)$ is a right-continuous irreducible Markov chain taking values in $\SM$ and $\CC_r$ is the phase space (see Section \ref{S2} for details). Since most of such systems cannot be solved explicitly, numerical solutions become a viable alternative.
Therefore, our main aim is to propose easily implementable explicit numerical schemes and to approximate the dynamical behaviours of \eqref{ISFDE} in finite and infinite horizons.

Over the past decade, HSFDEs have been extensively studied. For the finite delay case, Luo et al. \cite{LMS11} studied well-posedness, asymptotic boundedness and stability. Dang et al. \cite{DNN21} investigated stability in distribution, and Hairer et al. \cite{HMS11} focused on ergodicity. 
We can also see \cite{BS17, KS20} and the references therein for more developments on HSFDEs with finite delay.	
Recently, the theory of HSFDEswID has
received considerable attention. For the case of Lipschitz-type coefficients, the well-posedness has been established in \cite{LLMS23, LX21, WYM17}. Besides, Wu et al. \cite{WYM17} proved the existence and uniqueness of the invariant probability measure for SFDEs with infinite delay. Bao et al. \cite{BWY20} obtained the exponential ergodicity for neutral SFDEs with infinite delay in the Wasserstein distance. The theories on the ergodicity and asymptotic log-Harnack inequality were also investigated in \cite{BWY19, LX21, SWW22}. Moreover, for the case of non-Lipschitz coefficients, we refer the reader to \cite{WWYZ22} and the references therein.

Although significant progress has been made in the study of the HSFDEs  dynamical behaviors, most physical models are nonlinear and cannot be solved explicitly.
Therefore, developing efficient numerical methods to approximate their dynamics in finite and infinite horizons is of great importance. Up to now, the existing works on the numerical methods of HSFDEs focus on finite delay cases,  for instance, the Euler–Maruyama (EM) scheme \cite{BSY23, CDHS25, MY06, WM08}, the backward EM scheme \cite{SWMW24, ZX191}, the $\theta$-EM scheme \cite{CDHS24}, the truncated EM scheme \cite{GMY18, LMS24, SHGL22}, and the references therein. On the other hand,
Liu-Wang \cite{LW12} and Liu \cite{L15} employed the EM scheme and the Milstein scheme  to perform numerical experiments by computer to support their theoretical results on the persistence and extinction of the biological populations described by HSFDEswID without rigorous numerical analysis.
The numerical results on  HSFDEswID are still rare. Under the global Lipschitz condition, Li et al. \cite{LLMS23} studied a large class of HSFDEswID  (see \cite[(4.7)]{LLMS23} for details) and established the strong convergence rate of order 1/2 for the EM scheme. Recently, Zhou et al. \cite{ZXM25} proposed a truncated EM (TEM) scheme and established strong convergence with a rate less than $1/2$ for  the super-linear HSFDEswID. 
It is worth noting that both of them are restricted to systems with a specific form given in \cite[(4.7)]{LLMS23}.

Many important HSFDEwID models exhibit characteristics of not only super-linear coefficients but also highly complex structures \cite{K93, L15, LW12, YW22}. 
Thus, our main aim is to propose an easily implementable scheme for a large class of super-linear HSFDEswID with the generic form \eqref{ISFDE} and to establish the stability of the numerical solutions. 
Due to the super-linearity of the coefficients, the classical EM scheme fails to achieve strong convergence for super-linear HSFDEswID \cite{ZXM25}. Notably, the EM method fails even for classical stochastic differential equations (SDEs) \cite{HJK11}. Consequently, overcoming the super-linear growth of the drift and diffusion coefficients presents a significant challenge for constructing appropriate numerical schemes for \eqref{ISFDE}. 
The delay interval is an infinite horizon, and the corresponding phase space $\CC_r$ is an infinite-dimensional Banach space with a more intricate structure than that in finite delay systems and classical SDEs \cite{KS80, MN89, WYM17} (that is, $x_t$ contains an infinite amount of historical information and belongs to $\CC_r$). These two features make the computation of $x_t$ particularly difficult, thereby placing considerable demands on the design and analysis of numerical schemes.
Storage requirements represent a major computational challenge in the analysis of numerical methods for systems with memory \cite{LLS08, SLL06, ZX19}, especially for infinite-memory systems \cite{M12}. Intuitively, a longer delay interval requires the storage of more historical data. Therefore, it is highly desirable to construct numerical schemes that can approximate \eqref{ISFDE} in both finite and infinite horizons while requiring less storage.

In this paper, adapting both the space and time truncation techniques, we construct an easily implementable TEM scheme for a class of super-linear HSFDEswID and establish its strong convergence along with convergence rate. Furthermore, we develop a more refined numerical scheme to approximate the stability of the original SFDE. More precisely, the main contributions are as follows.
\begin{itemize}
\item 
To handle the super-linearity of drift and diffusion coefficients, we construct a well designed space-truncation mapping based on the growth rates of the coefficients,  thereby defining a TEM scheme. This scheme avoids possible large excursions caused by the super-linear growth of the coefficients, which allows us to establish its strong convergence in the $q$th moment.
Furthermore, the $1/2$-order convergence rate is established for the TEM numerical solution under certain additional polynomial growth conditions. 

\item 
To overcome the difficulties arising from infinite delay and the infinite-dimension of phase space, we truncate the infinite delay in the time horizon and construct an explicit numerical segment process $X^{k,\Delta}_{t_j}$ (with finite historical information) by the linear interpolation of $\{X^{k,\Delta}({t_{j+m}})\}_{-kl \le m \le 0}$.  We succeed to prove that the numerical segment process $X^{k,\Delta}_{t_j}$ converges to the exact $x_t$ (infinite historical information).

\item 
As mentioned above, the delay truncation device enables storing only $kl+1$ historical grid points, which reduces the storage cost to $O((kl+1)n)$ (see Section \ref{S3} for details).

\item 
We propose a more precise TEM scheme, whose numerical solution keeps the underlying stability of the HSFDEswID. This new TEM scheme only stores finite historical grid points with a storage cost $O((kl+1)n)$, which does not increase with time going by (see Section \ref{S5} for more details).
\end{itemize}

The rest of the paper is organized as follows. 
Section \ref{S2} introduces notation and preliminary results.
Section \ref{S3} proposes an explicit TEM scheme and establishes the strong convergence of the numerical solution.
Section \ref{S4} studies the convergence rate of the numerical solution under slightly stronger assumptions.
Section \ref{S5} proposes a refined scheme and investigates the numerical exponential stability.
Section \ref{S6} presents an example to illustrate our results.

\section{Preliminaries}\label{S2}
We begin with some notation used in this paper. Let $\NN=\{0, 1, \cdots\}$, $\NN_+=\{1, 2, \cdots\}$ and $d, n, N \in \NN_+$.
Let $\RR^n$ be the $n$-dimensional Euclidean space and $\BB(\RR^n)$ denote the Borel algebra on $\RR^n$. For $x \in \RR^n$, we denote by $|x|$ its Euclidean norm. If $A$ is a matrix, we denote by $A^\mathrm{T}$ the transpose of $A$ and let $|A|$ denote the trace norm, that is, $|A| = \sqrt{\trace(A^{\rm T} A)}$. 
Let $\RR_+ = [0, +\infty)$ and $\RR_- = (-\infty, 0]$.
If $a, b \in \RR$, define $a \wedge b = \min{a, b}$, $a \vee b = \max{a, b}$, and let $\lfloor a \rfloor$ denote the integer part of $a$. 
For $c = (c_1, \cdots, c_N)$, define $\c c = \max_{1\le i \le N} c_i$ and $\hat c = \min_{1\le i \le N} c_i$.
If $\DD$ is a set, its indicator function is denoted by $\II_\DD$, namely, $\II_\DD(x) = 1$ if $x \in \DD$ and $0$ otherwise. Throughout the paper, let $C$ denote a generic positive constant whose value may vary from line to line while we use $C_\zeta$  to emphasize the dependence on the parameter $\zeta$.

Let $(\W, \CF, \PP)$ be a complete probability space equipped with a filtration $\{\CF_t\}_{t\ge 0}$ satisfying the usual conditions (i.e., it is right continuous and increasing while $\CF_0$ contains all $\PP$-null sets), and $\EE$ denote the expectation corresponding to $\PP$.
Let $\{B(t)\}_{t\ge 0}$ be a $d$-dimensional Brownian motion defined on this probability space. Suppose that $\{\o(t)\}_{t\ge 0}$ is a right-continuous irreducible Markov chain taking values in the state space $\SM = \{1, 2, \cdots , N\}$, and independent of the Brownian motion $\{B(t)\}_{t \ge 0}$.
The generator of $\{\o(t)\}_{t\ge 0}$ is denoted by $Q = (q_{ij})_{N \K N}$ such that for any $\de > 0$,
$$
\PP(\o(t+\de)=j|\o(t)=i) =
\begin{cases}
q_{ij}\de + o(\de), \quad &\text{ if } i \neq j, \\
1 + q_{ii}\de + o(\de), \quad &\text{ if } i = j,
\end{cases}
$$
where $o(\de)$ satisfies $\lim_{\de \ra 0} o(\de)/\de = 0$. Here $q_{ij} \ge 0$ is the transition rate from state $i$ to state $j$ if $i \neq j$, while $q_{ii} = -\sum_{j \neq i} q_{ij}$. It is well known that almost every sample path of $\o(t)$ is a right-continuous step function with  finite jumps in any finite interval of $\RR_+$.

Denote by $C(\RR_-; \RR^n)$ the set of continuous functions $\p : \RR_-\rightarrow\RR^n$. Fix a positive number $r$ and define the phase space $\CC_r$ with fading memory by
\begin{equation}
\CC_r =  \bigg\{ \p\in C(\RR_-; \RR^n): \|\p\|_r:=\sup_{-\8<u\le0}e^{ru}|\p(u)| <\8 \bigg\}.
\end{equation}
One observes that $(\CC_r, \|\cdot\|_r)$ is a Banach space (see \cite{HMN91} for more details).
Moreover, denote by $\CP_0$ the family of probability measures $\m$ on $\RR_-$. For each $a>0$, define
$$
\CP_a = \bigg\{ \m \in \CP_0: \int_{-\8}^0 e^{-au} \m(\d u)<\8 \bigg\}.
$$
Furthermore, let $\mu^{(a)} := \int_{-\8}^0 e^{-a u} \mu(\d u)$
for each $\mu\in \CP_a$. Clearly, $\CP_{a_1} \subset \CP_{a_2}\subset \CP_0$ for $a_1 >a_2 >0$.  Moreover, for any $\mu\in \CP_{a_1}$, $\mu^{(a)}$
is strictly increasing and continuous with respect to $a \in [0,a_1]$.

For the regularity and the $p$th moment boundedness of the exact solution of \eqref{ISFDE}, we impose the following assumptions.
\begin{assp}\label{Lip}
There exists a  measure $\mu_1 \in \CP_{2r}$ and
a constant $K_R > 0$ such that
\begin{equation*}
|f(\p, i)-f(\f, i)|^2 \vee |g(\p, i)-g(\f, i)|^2  \le
  K_R  \int_{-\infty}^0 |\p(u)-\f(u)|^2\mu_1(\d u) 
\end{equation*}
for $i \in \SM$ and $\p, \f\in \CC_r$ with $\|\p\|_r\vee\|\f\|_r \le R$ for any $R>0$.
\end{assp}

\begin{assp}\label{mon}
There exist an  integer $p \ge 1$, constants $\bar p \ge 2p$, $a_1>0$, $a_2 \ge 0$
and $\m_2 \in \CP_{2pr}$, $\m_3 \in \CP_{\bar p r}$ such that for any $\p \in \CC_r$ and $i \in \mathbb{S}$,
$$
\begin{aligned}
     & |\phi(0)|^{2(p-1)} \left( 2 \langle \p(0), f(\p,i) \rangle  + (2p-1)|g(\p,i)|^2 \right) \\
 \le &  a_1\left(1 +\int_{-\infty}^0 |\p(u)|^{2p} \mu_2(\d u)\right) - a_2|\p(0)|^{\bar p} + a_2 \int_{-\8}^0 |\p(u)|^{\bar p} \m_3(\d u).
\end{aligned}
$$
\end{assp}

Now we prepare the well-posedness for the exact solution of \eqref{ISFDE}, which can be proved in the similar way as \cite[Theorem 2.3]{LLMS23}.
\begin{theorem}\label{th2.3}
Let {\rm Assumptions} \ref{Lip} and \ref{mon} hold. Then HSFDEwID \eqref{ISFDE} with the initial data \eqref{2.3} has a unique global solution $x(t)$ for $t\in (-\infty, \infty)$ satisfying
\begin{equation}\label{BISFDE}
\sup_{0\le t\le T}\EE|x(t)|^{2p} \le C_{T}, \quad \forall T>0,
\end{equation}
\end{theorem}

To construct the appropriate numerical solution of the HSFDEwID \eqref{ISFDE}, we cite the approximation theory established in \cite{LLMS23}, which revealed that the solutions of HSFDEwID can be approximated by those of the time-truncated HSFDEs with finite delay information, whose segment processes only take values in finite interval. Precisely,
for any integer $k \ge 1$, define a time-truncation mapping $\pi_k:\CC_r\ra\CC_r$ by
$$
\pi_k(\f)(u)=
\begin{cases}
  \f(u), & \mbox{if } u \in [-k,0], \\
  \f\left(-k\right),  & \mbox{if }   u \in \left(-\8,-k\right).
\end{cases}
$$
Then define $f_k:\CC_r \K\SM \to \RR^n$ and $g_k:\CC_r\K\SM\to\RR^{n \K d}$ by
$$
f_k(\f,i)=f(\pi_k(\f),i),~
g_k(\f,i)=g(\pi_k(\f),i), ~\forall i\in \SM.
$$
Thus the time-truncated HSFDE is defined as
\begin{equation}\label{TSFDE}
\begin{aligned}
\d x^k(t) & =f_k(x^k_t,\o(t))\d t + g_k(x^k_t,\o(t))\d B(t) \\
\end{aligned}
\end{equation}
on $t \ge 0$ with initial data \eqref{2.3},  where $x^k_t(u) = x^k(t+u) $ for any $ u \le 0.$ 

According to \cite[Theorem $3.1$, Theorem $3.4$]{LLMS23}, one notices that the solution $x^k(t)$ of the time-truncated HSFDE \eqref{TSFDE} converges to $x(t)$ in the $q$th moment. 
\begin{lemma}\label{lemma3.1}{\bf \cite[Theorem $3.1$, Theorem $3.4$]{LLMS23}}
Let Assumptions \ref{Lip} and \ref{mon} hold. Then the time-truncated HSFDE \eqref{TSFDE} with initial data \eqref{2.3} has a unique global solution $x^k(t)$ on $t\in (-\infty, \infty)$. For any $k \ge 1$, define
\begin{equation}\label{rkh}
\tau^k_h = \inf\{t \ge 0: |x^k(t)|\ge h\}.
\end{equation} Then for any $T > 0$,
\begin{equation}\label{BTSFDE}
\sup_{k \ge 1} \sup_{0\le t\le T} \EE |x^k(t)|^{2p} \le C_{T,\xi},~~~~ 
\sup_{k \ge 1}\PP\{\tau^k_h \le T \} \le \frac{C_{T,\xi}}{h^{2p}}.
\end{equation}
Moreover, if $\mu_1 \in \CP_{b}$ with $b>2r$, then
\begin{equation}\label{con}
\lim_{k \to +\infty} \left( \sup_{0\le t\le T}\EE |x^k(t)-x(t)|^q \right) = 0,~ \forall 
q \in (0, 2p),~ T > 0.
\end{equation}
\end{lemma}

With the help of Lemma \ref{lemma3.1},  we may construct  an appropriate scheme for the corresponding time-truncated HSFDE \eqref{TSFDE} and then use it to numerically approximate \eqref{ISFDE}.

\section{Explicit scheme and convergence} \label{S3}
This section is devoted to constructing an explicit numerical scheme for HSFDEswID \eqref{ISFDE} and establishing its strong convergence theory.
This section first constructs an easily implementable numerical scheme for \eqref{TSFDE}, establishes the boundedness of the $2p$th moment of the numerical solution, and then proves the strong convergence of the numerical solution to the exact  of \eqref{TSFDE}. By the bridge of Lemma \ref{lemma3.1}, we further yield  the desired convergence of the numerical solution to the exact of \eqref{ISFDE}.

 Given a step size $\Delta > 0$, let $t_j = j \Delta$ and $\theta_j = \theta(t_j)$. Then $\{ \theta_j \}_{j \in \mathbb{N}}$ forms a discrete-time Markov chain with one-step transition probability matrix $e^{\Delta Q}$. This discrete Markov chain can be simulated by the technique  in \cite[p.~112]{MY06}.
To construct the scheme, we begin with  estimating the growth rates of $f$ and $g$.
Under Assumption \ref{Lip}, we  choose an increasing function $\G: \RR_+ \to \RR_+$ such that $\G(R) \to +\infty$ as $R \to +\infty$ 
and
\begin{equation}\label{G}
\begin{aligned}
& |f(\p,i)|^2  \le \G^2(R) \left( 1 + \int_{-\8}^0 |\p(u)|^2 \mu_1(\d u) \right), \\
& |g(\p,i)|^2 \le  \G(R) \left( 1 + \int_{-\8}^0 |\p(u)|^2 \mu_1(\d u) \right), \\
\end{aligned}
\end{equation}
for any $(\p, i) \in \CC_r \K \SM$ with $\|\p\|_r \le R$. 
Let $\G^{-1}$ be the inverse function of $\G$, where $\G^{-1}:[\G(0), +\infty)\to\RR_+$.   Note that there can be many functions $\G$ satisfying inequality \eqref{G}; here, finding just one is sufficient, as it does not affect the subsequent conclusions. 
Without loss of generality, assume that there exists a positive integer $l$ such that $\D=1/l\in (0, 1]$. Define a truncation
mapping $\Lambda^\Delta:\RR^n \to \RR^n$ by
{ \begin{align}\label{L}
\L^\D(x)=\bigg( |x|\we\G^{-1}( \D^{-\lbd})  \bigg)\frac{x}{|x|}, \quad \forall x\in\RR^n.
\end{align}}
where the parameter $\lbd \in (0,1/2]$ and $x/|x|=0$ if $x=0$.
Next, we propose a numerical scheme for \eqref{TSFDE} called the truncated Euler-Maruyama (TEM) method. For any integers $k\geq 1$ and  $j$,
define
{ \begin{align}\label{TEM}
\begin{cases}
Y^{k, \D}(t_j)=\xi(t_j), ~-kl \le j \le 0,\\
X^{k, \D}({t_j})=\L^{\D}(Y^{k, \D}(t_j)), ~ j \ge -kl, \\
Y^{k, \D}(t_{j+1})=X^{k, \D}(t_j)+f(X^{k, \D}_{t_j},\o_j)\D
+ g( X^{k, \D}_{t_j},\o_j)\D B_j, ~ j \ge 0, \\
\end{cases}
\end{align}}
where $\D B_j=B(t_{j+1})-B(t_j)$ and $X^{k, \D}_{t_j}$ is a $\CC_r$-valued random variable defined by
\begin{equation}\label{LI}
X^{k, \D}_{t_j}(u)=\\
\left\{\begin{aligned}
& \frac{t_{m+1}-u}{\D} X^{k, \D}(t_{j+m}) + \frac{u-t_m}{\D} X^{k, \D}(t_{j+m+1}),\\
 & ~~~~~~~~~~~~~~~t_m \le u \le t_{m+1},~-k l \le m\le -1,\\
& X^{k, \D}(t_{j-k l}),~u < -k,
\end{aligned}
\right.
\end{equation}
Define the continuous-time numerical solution processes 
\begin{equation}\label{conti-time-TEM}
X^{k, \D}(t)=X^{k, \D}(t_j),\ Y^{k, \D}(t) = Y^{k, \D}(t_j), ~ \forall t \in [t_j, t_{j+1}), ~ { j \ge -kl,}
\end{equation}
and the numerical segment process
\begin{equation}\label{n-seg}
X^{k,\Delta}_t = X^{k,\Delta}_{t_j}, \quad \forall t \in [t_j,t_{j+1}), j \in \NN.
\end{equation}
Here we highlight the features  of our TEM scheme.
\begin{itemize}
\item 
It follows from \eqref{LI} that for each $j \in \mathbb{N}$, we only need to store $\{X^{k,\D}(t_{j+m})\}_{m=-kl}^{0}$ in active memory to compute $X^{k,\D}(t_{j+1})$. Consequently, the TEM scheme is suitable for the  computer operation and has only an $O((kl+1)n)$ storage cost despite involving an infinite delay.

\item 
The numerical solution $X^{k, \D}(t_j)$ is obtained by truncating the intermediate term $Y^{k, \D}(t_j)$ in the space $\RR^n$ to avoid its possible large excursions due to the nonlinearities of the coefficients and the Brownian motion increments. Moreover, from \eqref{G} one observes that
{ \begin{equation}\label{linear}
\begin{aligned}
\big| f(X^{k, \D}_{t_j},i) \big|^2 &\le \D^{-2\lbd}\left( 1 + \int_{-\8}^0 \big| X^{k, \D}_{t_j}(u) \big|^2  \mu_1(\d u) \right),\\
\big| g(X^{k, \D}_{t_j},i) \big|^2  &\le \D^{-\lbd} \left( 1 + \int_{-\8}^0 \big| X^{k, \D}_{t_j}(u) \big|^2 \mu_1(\d u) \right)
\end{aligned}
\end{equation}}
for any $j \in \NN$ and $i \in \SM$.

\item 
If $f$ and $g$ are globally Lipschitz continuous, which implies $K_R \equiv K$ for any $R > 0$, then let $\G(R) \equiv K$ and $\G^{-1}(R)=\8$. Thus, $\L^{\D}(x)=x$ for any $\Delta \in (0, 1]$, $x\in\RR^n$.  Obviously, the TEM scheme degenerates to the classical EM scheme.
\end{itemize}

\subsection{Moment boundedness }\label{s3.1}
By the similar argument as in \cite[Theorem 3.1]{SHGL22} one obtains $$\sup_{0 < \Delta \le 1}\sup_{0 \le t \le T} \EE \big|X^{k,\D}(t) \big|^{2p} \leq C_k.$$  Here the upper bound  depends on $k$. However we need the bound   independent of $k$ in order for the convergence analysis of the numerical solution. Therefore, we give the uniform boundedness of the TEM numerical solution.

\begin{theorem}\label{th-BTN}
Let Assumptions \ref{Lip} and \ref{mon} hold. Then 
\begin{equation}\label{BTN}
\sup_{k \ge 1} \sup_{0 < \Delta \le 1}\sup_{0 \le t \le T} \EE \big|X^{k,\D}(t) \big|^{2p} \le C_T (1+\|\xi\|_r^{2p} + \|\xi\|_r^{\bar p}), \quad \forall T>0.
\end{equation}
\end{theorem}

\begin{proof}
We focus on the case $p > 2$. By the similar arguments we can prove the desired results for $p=1$ and  $p=2$.
For any $k \ge 1$, $\D \in (0,1]$ and $j \in \NN$, let 
\begin{equation}\label{eq*1}
f^{k,\D}_j = f( X^{k,\D}_{t_j},\o_j),~~~~g^{k,\D}_j = g(X^{k,\D}_{t_j},\o_j).\end{equation}
{ By the third equality of \eqref{TEM},}
\begin{equation*}
\left|Y^{k,\D}(t_{j+1})\right|^2 = \left|X^{k,\D}(t_j)\right|^2 + \eta^{k,\D}_j,
\end{equation*}
where
$$
\eta^{k,\D}_j = \big|f^{k,\D}_j\big|^2\D^2 + \big|g^{k,\D}_j \D B_j\big|^2 +2 \big\langle X^{k,\D}(t_j), f^{k,\D}_j \big\rangle \D
+  2 \big\langle X^{k,\D}(t_j) + f^{k,\D}_j \D,  g^{k,\D}_j \D B_j \big\rangle.
$$
{ By virtue of the second equality of \eqref{TEM},} we derive that
$$
\big|X^{k,\D}(t_{j+1})\big|^{2p} \le \left|Y^{k,\D}(t_{j+1})\right|^{2p} \le  \left( \left|X^{k,\D}(t_j)\right|^2 + \eta^{k,\D}_j \right)^p.
$$
Using the binomial expansion theorem and then taking the conditional expectation with respect to $\CF_{t_j}$ on both sides yield
\begin{equation}\label{b}
\begin{aligned}
    \EE\left( \left|X^{k,\D}(t_{j+1})\right|^{2p} \big| \CF_{t_j}\right) 
  =   \left|X^{k,\D}(t_j)\right|^{2p} + \sum_{m=1}^p C_p^m  \EE \left( \left|X^{k,\D}(t_j)\right|^{2(p-m)}  \left(\eta^{k,\D}_j\right)^m \big| \CF_{t_j} \right),
\end{aligned}
\end{equation}
where $C_p^m = p!/\left(  m! (p-m)! \right)$.
Since $\D B_j$ is independent of $\CF_{t_j}$, it follows that
$$
\EE (\D B_j|\CF_{t_j}) = 0, ~ \EE(|A\D B_j|^2) = |A|^2\D, ~ \forall A \in \RR^{n \K d}.
$$
This, together with \eqref{linear} and Young's inequality, gives
\begin{equation}\label{t1}
\begin{aligned}
   & \EE \left( \left|X^{k,\D}(t_j)\right|^{2(p-1)}  \eta^{k,\D}_j \big| \CF_{t_j} \right) \\
   = & \left|X^{k,\D}(t_j)\right|^{2(p-1)} \left[ \left( 2 \big\langle X^{k,\D}(t_j),  f^{k,\D}_j \big\rangle + \big| g^{k,\D}_j \big|^2 \right)\D + \big| f^{k,\D}_j \big|^2\D^2 \right] \\
   \le & \left|X^{k,\D}(t_j)\right|^{2(p-1)}  \left( 2 \big\langle X^{k,\D}(t_j),  f^{k,\D}_j \big\rangle + \big| g^{k,\D}_j \big|^2 \right)\D \\
   +& C\D \left( 1 + \big|X^{k,\D}(t_j)\big|^{2p} + \left( \int_{-\8}^0 \big|X^{k,\D}_{t_j}(u)\big|^2 \m_1(\d u) \right)^p  \right).
\end{aligned}
\end{equation}
Notice that
$$
\begin{aligned}
    & \mathbb{E}|A\Delta B_j|^{2m} \le (2m-1)!! n
    ^m |A|^{2m} \Delta^m,\ \EE\left[\left( \left\langle z, A \Delta B_j \right\rangle \right)^{2m-1}|\CF_{t_j} \right] = 0,\\
    & \EE\left[\left( \left\langle z, A \Delta B_j \right\rangle \right)^{2m}|\CF_{t_j} \right] = (2m-1)!!|\left\langle z, A\right\rangle|^{2m} \Delta^m,\ \forall m\ge 2
\end{aligned}
$$
for any $z \in \mathbb{R}^n$ and $A \in \mathbb{R}^{n\times d}$, where $(2m-1)!!=1\cdot 3 \cdots (2m-1)$. 
Combining these with \eqref{linear} and Young's inequality leads to
\begin{equation}\label{t2}
\begin{aligned}
     &\EE \left( \left|X^{k,\D}(t_j)\right|^{2(p-2)}  \left(\eta^{k,\D}_j\right)^2 \Big| \CF_{t_j} \right) \\
  \le & \left|X^{k,\D}(t_j)\right|^{2(p-2)}  \EE \Big[\Big( \big|f^{k,\D}_j\big|^4\D^4 + \big|g^{k,\D}_j\D B_j\big|^4
+ 4\big|\big\langle X^{k,\D}(t_j), f^{k,\D}_j\big\rangle\big|^2  \D^2 \\
  & + 4\big|\big\langle X^{k,\D}(t_j), g^{k,\D}_j\D B_j \big\rangle\big|^2 + 4\big|\big\langle f^{k,\D}_j, g^{k,\D}_j\D B_j \big\rangle\big|^2 \D^2 + 2\big|f^{k,\D}_j\big|^2 \big|g^{k,\D}_j\big|^2 \big|\D B_j\big|^2 \D^2  \\
& + 4\big|X^{k,\D}(t_j)\big|\big|f^{k,\D}_j\big|^3\D^3
+  12\big| X^{k,\D}(t_j) \big| \big| f^{k,\D}_j \big| \big|g^{k,\D}_j\D B_j\big|^2 \D \Big) \Big|\CF_{t_j} \Big]\\
\le & 4 \left|X^{k,\D}(t_j)\right|^{2(p-1)} \big|g^{k,\D}_j\big|^2 \D + C\D \left( 1 + \big|X^{k,\D}(t_j)\big|^{2p} + \left( \int_{-\8}^0 \big|X^{k,\D}_{t_j}(u)\big|^2 \m_1(\d u) \right)^p  \right).
\end{aligned}
\end{equation}
Similarly, for any $m=3, \cdots, p$, we have
\begin{equation}\label{t3}
\begin{aligned}
& \EE \left( \left|X^{k,\D}(t_j)\right|^{2(p-m)}  \left(\eta^{k,\D}_j\right)^m \big| \CF_{t_j} \right)  \\
\le & C\D \left( 1 + \big|X^{k,\D}(t_j)\big|^{2p} + \left( \int_{-\8}^0 \big|X^{k,\D}_{t_j}(u)\big|^2 \m_1(\d u) \right)^p  \right).
\end{aligned}
\end{equation}
Inserting \eqref{t1}, \eqref{t2} and \eqref{t3} into \eqref{b} and then using Assumption \ref{mon} provides
\begin{align}\nn \label{t}
   & \EE\left( \left|X^{k,\D}(t_{j+1})\right|^{2p} \big| \CF_{t_j}\right) \\ \nn
\le  & \left|X^{k,\D}(t_j)\right|^{2p} + p \left|X^{k,\D}(t_j)\right|^{2(p-1)} \left( 2 \big\langle X^{k,\D}(t_j), f^{k,\D}_j \big\rangle + (2p-1)\big|g^{k,\D}_j\big|^2  \right) \D \\ \nn
& +  C\D \left( 1 + \big|X^{k,\D}(t_j)\big|^{2p} + \left( \int_{-\infty}^0 \big|X^{k,\D}_{t_j}(u)\big|^2 \m_1(\d u) \right)^p  \right) \\
\le  & \left|X^{k,\D}(t_j)\right|^{2p} + p a_1 \D \left( 1+ \int_{-\8}^0 \big|X^{k,\D}_{t_j}(u)\big|^{2p} \m_2(\d u) \right) \\ \nn
& -  p a_2 \D \left( \left|X^{k,\D}(t_j)\right|^{\bar p} - \int_{-\8}^0 \big|X^{k,\D}_{t_j}(u)\big|^{\bar p} \m_3(\d u) \right) \\ \nn
& +  C\D \left( 1 + \big|X^{k,\D}(t_j)\big|^{2p} + \left( \int_{-\8}^0 \big|X^{k,\D}_{t_j}(u)\big|^2 \m_1(\d u) \right)^p  \right).
\end{align}
Taking expectations in the inequality above and summing from $j=0$ to $n-1$, where $1 \le n \le \lfloor T/\D\rfloor$ (without loss of generality, we assume that $T \ge \D$), we obtain
\begin{equation}\label{diedai}
\begin{aligned}
    \EE \left|X^{k,\D}(t_n)\right|^{2p} 
\le  & \|\xi\|_r^{2p} + CT + C\D \sum_{j=0}^{n-1}  \EE \big|X^{k,\D}(t_j)\big|^{2p}+  I_1(n) \\
+  & I_2(n)
-  p a_2 \D \sum_{j=0}^{n-1} \EE \left|X^{k,\D}(t_j)\right|^{\bar p} + I_3(n) ,
\end{aligned}
\end{equation}
where
$$
\begin{aligned}
& I_1(n) =  C\D \sum_{j=0}^{n-1}\EE \left( \int_{-\8}^0 \big|X^{k,\D}_{t_j}(u)\big|^2 \m_1(\d u) \right)^p,\\
      & I_2(n) = p a_1 \D \sum_{j=0}^{n-1} \EE \int_{-\8}^0 \big|X^{k,\D}_{t_j}(u)\big|^{2p} \m_2(\d u),\\
      & I_3(n) = p a_2 \D \sum_{j=0}^{n-1} \EE \int_{-\8}^0 \big|X^{k,\D}_{t_j}(u)\big|^{\bar p} \m_3(\d u).
\end{aligned}
$$
Obviously,
\begin{equation}\label{Int}
\begin{aligned}
I_1(n) =   C\D \sum_{j=0}^{n-1}\EE \left( \int_{-\8}^{-t_j} \big|X^{k,\D}_{t_j}(u)\big|^2 \m_1(\d u) + \int_{-t_j}^0 \big|X^{k,\D}_{t_j}(u)\big|^2 \m_1(\d u) \right)^p.
\end{aligned}
\end{equation}
Let $\mathbb{A} = \{0, 1, \cdots, kl\}$
and $\mathbb{A}^c = \mathbb{N}/\mathbb{A}$.
If $j \in \{0, \cdots, n-1\} \cap \mathbb{A} $, 
$$
       \int_{-\8}^{-t_j} \big|X^{k,\D}_{t_j}(u)\big|^2 \m_1(\d u) 
      =  \int_{-\8}^{-k} \big|X^{k,\D}_{t_j}(-k)\big|^2 \m_1(\d u) + \int_{-k}^{-t_j} \big|X^{k,\D}_{t_j}(u)\big|^2 \m_1(\d u).
$$
{ Using \eqref{LI}, Jensen's inequality (\cite[p. 53]{MY06}) and the fact that $|\Lambda^{\Delta}(x)| \le |x|$,} we compute for any $j \in \{0, \cdots, n-1\} \cap \mathbb{A}$,
\begin{align} \nn
      &  \int_{-k}^{-t_j} \big|X^{k,\D}_{t_j}(u)\big|^2 \m_1(\d u)  \\  \nonumber
      \le & \sum_{m=-k l}^{-j-1} \int_{t_m}^{t_{m+1}} \left( \frac{t_{m+1}-u}{\D} |\xi(t_{j+m})|^2  + \frac{u-t_m}{\D} |\xi(t_{j+m+1})|^2 \right)  \m_1(\d u) \\  \nonumber
      \le & \|\xi\|_r^2 e^{-2rt_j} \sum_{m=-k l}^{-j-1} \int_{t_m}^{t_{m+1}} \left( \frac{t_{m+1}-u}{\D} e^{-2r t_m}  + \frac{u-t_m}{\D} e^{-2r t_{m+1}} \right) \m_1(\d u) \\  \nn
      \le &  \|\xi\|_r^2 e^{-2rt_j + 2r\D}   \sum_{m=-k l}^{-j-1} \int_{t_m}^{t_{m+1}} e^{-2r t_{m+1}} \mu_1(\d u).
\end{align}
Similarly,
\begin{equation}\label{3.20}
      \int_{-\infty}^{-k} \big|X^{k,\D}_{t_j}(-k)\big|^2 \m_1(\d u) \le \|\xi\|_r^2 e^{-2rt_j} \int_{-\infty}^{-k} e^{-2ru} \m_1(\d u), \quad \forall j \in \{0, \cdots, n-1\} \cap \mathbb{A}.
\end{equation}
Consequently, it follows from $\mu_1 \in \CP_{2r}$ that
\begin{equation*}
\II_\mathbb{A}(j) \int_{-\infty}^{-t_j} \big|X^{k,\D}_{t_j}(u)\big|^2 \m_1(\d u)  \le e^{2r} \int_{-\infty}^{-t_j} e^{-2ru} \mu_1(\d u) \|\xi\|_r^2 
\le  e^{2r} \m_1^{(2r)} \|\xi\|_r^2.
\end{equation*}
On the other hand, if $j \in \{0, \cdots, n-1\} \cap \mathbb{A}^c$, we have
$$
\begin{aligned}
       \II_{\mathbb{A}^c}(j) \int_{-\infty}^{-t_j} \big| X^{k,\D}_{t_j}(u) \big|^2 \mu_1(\d u)   =  \II_{\mathbb{A}^c}(j) \mu_1\big((-\infty, -t_j]\big)  \big| X^{k,\Delta}(t_j-k) \big|^2.
\end{aligned}
$$
Hence, for any $0 \le j \le n-1$, 
\begin{equation*}
      \int_{-\infty}^{-t_j} \big|X^{k.\D}_{t_j}(u)\big|^2 \m_1(\d u) \le \|\xi\|_r^2 e^{2r} \m_1^{(2r)}  +   \II_{\mathbb{A}^c}(j) \mu_1\big((-\infty, -t_j]\big)  \big| X^{k,\Delta}(t_j-k) \big|^2,
\end{equation*}
Substituting this inequality into \eqref{Int} gives
\begin{equation*}
\begin{aligned}
   I_1(n)
   & \le     CT \|\xi\|_r^{2p}  + C\D \sum_{j=0}^{n-1}  \mu_1\big((-\infty, -t_j]\big)  \EE \big| X^{k,\Delta}(t_j-k) \big|^{2p} \II_{\mathbb{A}^c}(j)  \\
   & +   C\D \sum_{j=0}^{n-1} \EE \int_{-t_j}^0 \big|X^{k,\D}_{t_j}(u)\big|^{2p} \m_1(\d u). \\
\end{aligned}
\end{equation*}
Moreover, if $n-1 \in \mathbb{A}$, then $\II_{\mathbb{A}^c}(j)=0$ for any $j=0, \cdots, n-1$. Therefore,
\begin{equation}\label{3.23}
\begin{aligned}
   &  \sum_{j=0}^{n-1}  \mu_1\big((-\infty, -t_j]\big)  \EE \big| X^{k,\Delta}(t_j-k) \big|^{2p} \II_{\mathbb{A}^c}(j) \\
   = & \II_{\mathbb{A}^c}(n-1) \sum_{j=0}^{n-1}  \mu_1\big((-\infty, -t_j]\big)  \EE \big| X^{k,\Delta}(t_j-k) \big|^{2p}  \II_{\mathbb{A}^c}(j) \\
   = & \II_{\mathbb{A}^c}(n-1) \sum_{j=kl+1}^{n-1}  \mu_1\big((-\infty, -t_j]\big)  \EE \big| X^{k,\Delta}(t_j-k) \big|^{2p} \\
    = & \II_{\mathbb{A}^c}(n-1) \sum_{j=1}^{n-1-kl}  \mu_1\big((-\infty, -t_{j+kl}]\big)  \EE \big| X^{k,\Delta}(t_j) \big|^{2p}\\
    \le & \II_{\mathbb{A}^c}(n-1)  \sum_{j=0}^{n-1}  \mu_1\big((-\infty, -t_{j+kl}]\big)  \EE \big| X^{k,\Delta}(t_j) \big|^{2p}. \\
\end{aligned}
\end{equation}
It follows that 
\begin{equation}\label{Int'}
\begin{aligned}
   I_1(n)
   & \le     CT \|\xi\|_r^{2p}  + C\D  \II_{\mathbb{A}^c}(n-1)   \sum_{j=0}^{n-1}  \mu_1\big((-\infty, -t_{j+kl}]\big)  \EE \big| X^{k,\Delta}(t_j) \big|^{2p}  \\
   & +   C\D \sum_{j=0}^{n-1} \EE  \int_{-t_j}^0 \big|X^{k,\D}_{t_j}(u)\big|^{2p} \m_1(\d u). \\
\end{aligned}
\end{equation}
Moreover, when $n-1\in \mathbb{A}$, we have $t_j \le t_{n-1} \le k$ for any $0 \le j \le n-1$, which implies $[-t_j, 0] \subset [-k, 0]$. Therefore, by \eqref{LI} and Jensen's inequality, if $n-1 \in\mathbb{A}$, then
\begin{equation}\label{3.18}
\begin{aligned}
     &  \sum_{j=0}^{n-1} \int_{-t_j}^0 \big|X^{k,\D}_{t_j}(u)\big|^{2p} \m_1(\d u) \\
  = & \sum_{j=1}^{n-1} \sum_{m=-j}^{-1} \int_{t_m}^{t_{m+1}} \big|X^{k,\D}_{t_j}(u)\big|^{2p} \m_1(\d u) 
  =  \sum_{m=-n+1}^{-1} \sum_{j=-m}^{n-1} \int_{t_m}^{t_{m+1}} \big|X^{k,\D}_{t_j}(u)\big|^{2p} \m_1(\d u) \\
  \le  & \sum_{m=-n+1}^{-1} \sum_{j=-m}^{n-1} \int_{t_m}^{t_{m+1}} \left(  \frac{t_{m+1}-u}{\D} \big|X^{k,\D}(t_{j+m}) \big|^{2p} +  \frac{u-t_m}{\D} \big|X^{k,\D}(t_{j+m+1})\big|^{2p} \right)  \m_1(\d u) \\
  \le  & \sum_{m=-n+1}^{-1}  \int_{t_m}^{t_{m+1}} \frac{t_{m+1}-u}{\D} \left(\sum_{j=0}^{n-1+m}\big|X^{k,\D}(t_j) \big|^{2p}\right) +  \frac{u-t_m}{\D} \left(\sum_{j=1}^{n+m}\big|X^{k,\D}(t_{j+1})\big|^{2p}\right)  \m_1(\d u) \\
 \le  & \sum_{m=-n+1}^{-1}  \int_{t_m}^{t_{m+1}} \left(\sum_{j=0}^{n-1}\big|X^{k,\D}(t_j) \big|^{2p}\right)   \m_1(\d u).
\end{aligned}
\end{equation}
which implies
\begin{equation}\label{Int2-c1}
      \II_\mathbb{A}(n-1) \sum_{j=0}^{n-1} \int_{-t_j}^0 \big|X^{k,\D}_{t_j}(u)\big|^{2p} \m_1(\d u) \le 
      \II_\mathbb{A}(n-1) \sum_{j=0}^{n-1}\big|X^{k,\D}(t_j) \big|^{2p}.
\end{equation}
On the other hand,
\begin{equation}\label{Int2-c2'}
       \II_{\mathbb{A}^c}(n-1) \sum_{j=0}^{n-1} \int_{-t_j}^0 \big| X^{k,\D}_{t_j}(u) \big|^{2p} \mu_1(\d u) \le  \II_{\mathbb{A}^c}(n-1) \bigg( J^{k,\D}_1(n) + J^{k,\D}_2(n) \bigg)
\end{equation}
where
$$
      J^{k,\D}_1(n) = \sum_{j=0}^{kl-1} \int_{-t_j}^0 \big| X^{k,\D}_{t_j}(u) \big|^{2p} \mu_1(\d u)  +  \sum_{j=kl}^{n-1} \int_{-k}^0 \big| X^{k,\D}_{t_j}(u) \big|^{2p} \mu_1(\d u)
$$
and 
$$
      J^{k,\D}_2(n) = \sum_{j=kl}^{n-1} \int_{-t_j}^{-k} \big| X^{k,\D}_{t_j}(u) \big|^{2p} \mu_1(\d u).
$$
Noting that $[-t_j, 0] \subset [-k, 0]$ for any $0 \le j \le kl-1$ and employing the techniques used in the proof of  \eqref{Int2-c1}, one derives that
\begin{equation*}
\begin{aligned}
      \II_{\mathbb{A}^c}(n-1) J^{k,\D}_1(n)
      \le  \II_{\mathbb{A}^c}(n-1) \mu_1\big((-k,0)\big) \sum_{j=0}^{n-1} \big| X^{k,\D}(t_j) \big|^{2p}.
\end{aligned}
\end{equation*}
Similar to \eqref{3.23}, one can also derive that
\begin{equation*}
     \II_{\mathbb{A}^c}(n-1)  J^{k,\D}_2(n)  \le  \II_{\mathbb{A}^c}(n-1) \sum_{j=0}^{n-1}  \mu_1\big( (-t_{j+kl}, -k)\big)  \big| X^{k,\D}(t_j) \big|^{2p}.
\end{equation*}
Hence, in \eqref{Int2-c2'},
\begin{equation}\label{Int2-c2}
\begin{aligned}
        \II_{\mathbb{A}^c}(n-1) \sum_{j=0}^{n-1} \int_{-t_j}^0 \big| X^{k,\D}_{t_j}(u) \big|^{2p} \mu_1(\d u)  \le \II_{\mathbb{A}^c}(n-1)  \sum_{j=0}^{n-1}  \mu_1\big( (-t_{j+kl}, 0)\big)  \big| X^{k,\D}(t_j) \big|^{2p}.
\end{aligned}
\end{equation}
It follows from \eqref{Int2-c1} and \eqref{Int2-c2} that we always have 
\begin{equation}\label{Int2}
\begin{aligned}
       & \sum_{j=0}^{n-1} \int_{-t_j}^0 \big| X^{k,\D}_{t_j}(u) \big|^{2p} \mu_1(\d u)  \\
      \le &  \II_\mathbb{A}(n-1) \sum_{j=0}^{n-1}\big|X^{k,\D}(t_j) \big|^{2p} + \II_{\mathbb{A}^c}(n-1)  \sum_{j=0}^{n-1}  \mu_1\left( (-t_{j+kl}, 0)\right)  \big| X^{k,\D}(t_j) \big|^{2p}.
\end{aligned}
\end{equation}
Substituting \eqref{Int2} into \eqref{Int'} yields
\begin{equation}\label{I}
   I_1(n) \le CT \|\xi\|_r^{2p} + C\D \sum_{j=0}^{n-1} \EE \big|X^{k,\D}(t_j) \big|^{2p}.
\end{equation}
By the techniques we have used in the estimation of $I_1(n)$ and $\mu_2 \in \CP_{2pr}$, $\mu_3 \in \CP_{\bar p r}$, as well as \cite[Theorem 3.1]{LMS24}, we can show that
\begin{equation}\label{I23}
\begin{aligned}
      I_2(n) & \le CT \|\xi\|_r^{2p} + pa_1\D \sum_{j=0}^{n-1} \EE \big|X^{k,\D}(t_j) \big|^{2p}, \\
      I_3(n) & \le CT \|\xi\|_r^{\bar p} + pa_2\D \sum_{j=0}^{n-1} \EE \big|X^{k,\D}(t_j) \big|^{\bar p}.
\end{aligned}
\end{equation}
Hence, in \eqref{diedai},
\begin{equation}\label{bou-final}
  \EE \left|X^{k,\D}(t_n)\right|^{2p}
\le   CT (1+\|\xi\|_r^{2p} + \|\xi\|_r^{\bar p}) + C\D \sum_{j=0}^{n-1}  \EE \big|X^{k,\D}(t_j)\big|^{2p}.
\end{equation}
Applying the discrete Gronwall lemma, we obtain
$$
\EE \left|X^{k,\D}(t_n)\right|^{2p}   \le   CT (1+\|\xi\|_r^{2p} + \|\xi\|_r^{\bar p}) e^{CT}.
$$
The desired assertion follows from the fact that $C$ is independent of $k$ and $\Delta$.
The proof is complete.
\end{proof}
$\hfill\square$

The following corollary follows   directly from Theorem \ref{th-BTN} by a similar argument as \cite[Remark 5.5]{LLMS23}.
\begin{coro}\label{coro3.3}
   Let the conditions in Theorem \ref{th-BTN} hold. Then for any $\bar q \le 2p$ and $\mu \in \CP_{\bar q r}$, 
\begin{equation*}
\sup_{k \ge 1} \sup_{0 < \Delta \le 1}\sup_{0 \le t \le T} \EE \bigg( \int_{-\infty}^0  \big|X^{k,\D}_{t}(u) \big|^{\bar q} \mu(\d u)\bigg)^{\frac{2p}{\bar q}}  \le C_T (1+\|\xi\|_r^{2p} + \|\xi\|_r^{\bar p}), \quad \forall T > 0,
\end{equation*}
where $X^{k,\D}_t$ is defined in \eqref{n-seg}.
\end{coro}

\subsection{Strong convergence}
This subsection aims to establish the strong convergence of the TEM numerical solution $X^{k,\D}(t)$ to the exact $x(t)$ of \eqref{ISFDE}. 
Although \cite{LLMS23} gave the strong convergence of the EM numerical solution to the exact one of \eqref{ISFDE} under global Lipschitz condition,  the proof technique is not suitable for the TEM scheme due to the truncation mapping involved and the nonlinear coefficients. 
To overcome the obstacles, we make use of \eqref{con} in Lemma \ref{lemma3.1} and go to  analyze $\EE |x^k(t)-X^{k,\D}(t)|^q$. In order to deal with the  discontinuity of $X^{k,\D}(t)$ respect to $t$ we introduce an auxiliary process $Z^{k,\D}(t)$ and turn to estimate the bounds of $\EE |x^k(t)-Z^{k,\D}(t)|^q$ and  $\EE|Z^{k,\D}(t)-X^{k,\D}(t)|^q$. Then the desired result follows from the triangle inequality.  

Next we introduce the auxiliary  process
{ \begin{equation}\label{ap}
\left\{
\begin{aligned}
&Z^{k,\D}(t) = \xi(t), \quad t < 0,\\
&Z^{k,\D}(t) = X^{k,\D}(t_j)+f^{k,\D}_j(t-t_j)+g^{k,\D}_j(B(t)-B(t_j)),\ t \in [t_j, t_{j+1}),\ j \ge 0,\\
\end{aligned}\right.
\end{equation}}
where $f^{k,\D}_j$ and $g^{k,\D}_j$ are defined by \eqref{eq*1}.
Obviously,
\begin{equation}\label{XYZ}
Z^{k,\D}(t_j) = \lim_{t\to t_j^+}Z^{k,\D}(t)=X^{k,\D}(t_j), \quad \lim_{t\to t_j^-}Z^{k,\D}(t) = Y^{k,\D}(t_j),~\forall j \ge 0.
\end{equation}
One observes that $Z^{k,\D}(t_j)$ becomes a bridge connecting $ X^{k,\D}(t_j)$ and  $Y^{k,\D}(t_j)$.
For any $\Delta_1 \in (0,1]$ and $\Delta \in (0, \Delta_1]$, define
{ \begin{equation}\label{rh}
   \varrho^k_{\Delta,\Delta_1} = \inf \left\{t\ge0 : \big|Z^{k,\D}(t)\big| \ge \Gamma^{-1}(\Delta_1^{-\lambda}) \right\}.
\end{equation}}
{ Note that $Z^{k,\D}(\cdot)$ may be discontinuous on $(-\infty, T]$ but continuous on $(-\infty, \varrho^k_{\D,\Delta_1}\wedge T)$. Thus, we conclude that when $\varrho^k_{\Delta, \Delta_1} > 0$,
\begin{equation}\label{ap-con}
Z^{k,\D}(t)=
\xi(0) + \int_0^{t} f (X^{k,\D}_s, \o_s )\d s + \int_0^{t} g(X^{k,\D}_s, \o_s )\d B(s), \quad \forall t \in [0,\varrho^k_{\Delta, \Delta_1}]
\end{equation}
where $\theta_t = \theta_j$ for $t \in [t_j, t_{j+1})$ and $X^{k,\D}_t$ is defined in \eqref{n-seg}. Furthermore, one observes from \eqref{XYZ} and \eqref{rh} that, when $\varrho^k_{\Delta, \Delta_1} > 0$, we have
$$|Y^{k,\D}(t)| \le \Gamma^{-1}( \D_1^{-\lambda}), \quad \forall t \in [0, \varrho^k_{\Delta,\Delta_1}]. $$}
\begin{lemma}\label{l4.6}
Let the assumptions in Theorem \ref{th-BTN} hold. Then
$$
      \sup_{k \ge 1} \sup_{0 < \D \le 1} \sup_{0\le t\le T} \EE \big|Z^{k,\D}(t)\big|^{2p} \le C_{T,\xi}, \quad \forall T > 0.
$$
Moreover, for any $\Delta_1 \in (0, 1]$ and $\Delta \in (0,\Delta_1]$,
{ $$
      \sup_{k \ge 1} \mathbb{P}\{\varrho^k_{\Delta,\Delta_1}\le T\} \le \frac{C_{T,\xi}}{\big(\Gamma^{-1}( \Delta_1^{-\lambda})\big)^{2p}},\quad \forall T > 0.
$$}
\end{lemma}

\begin{proof}
We can obtain the boundedness using techniques similar to those in \cite[Lemma 3.5]{LMS24}. To avoid duplication, we omit the proof.
Fix $\Delta_1 \in (0,1]$ arbitrarily. For any $\D\in(0,\D_1]$  and $k \ge 1$, let $\beta^k_{\D,\D_1} = \lfloor \varrho^k_{\D,\D_1}/ \Delta \rfloor$. For simplicity, we write $\varrho = \varrho^k_{\D,\D_1}$ and $\beta = \beta^k_{\D,\D_1}$. For any $t \in [t_j, t_{j+1})$ with $j \in \mathbb{N}$, one has $\lfloor (t \wedge \varrho ) / \Delta \rfloor= j \wedge \beta$.
Clearly, $\varrho$ and $\beta$ are $\mathcal{F}_t$ and $\mathcal{F}_{t_j}$ stopping times.
Define
$$
Y^{k,\D}_{t_j}(u)=
\left\{\begin{aligned}
& \frac{t_{m+1}-u}{\D} Y^{k,\D}(t_{j+m}) + \frac{u-t_m}{\D} Y^{k,\D}(t_{j+m+1}), \\
&~~~~~~~~~~~~~~~~t_m\le u \le t_{m+1},\ -k l \le m \le -1,\\
& Y^{k,\D}(t_{j-k l}),~-\infty<u < -k,
\end{aligned}\right.
$$
{ where $Y^{k,\D}(t_j)$ is defined in \eqref{TEM}.}
For $\omega\in\{\beta \ge j+1\}$, we have
{ $
\big|Y^{k,\D}(t_j)\big| < \G^{-1}(\D_1^{-\lbd}) \le \Gamma^{-1}( \Delta^{-\lambda})
$}
and $X^{k,\D}(t_m) = Y^{k,\D}(t_m)$ for any $m \leq j$, which implies that
$Y^{k,\D}_{t_j}(\cdot) = X^{k,\D}_{t_j}(\cdot)$.
Then it follows from \eqref{TEM} that
\begin{align}\nn
Y^{k,\D}(t_{(j+1)\wedge\beta}) 
= Y^{k,\D}(t_{j\wedge\beta}) + \big( F^{k,\D}_j\D + G^{k,\D}_j\D B_j \big) \mathbf{1}_{[[0,\beta]]}(j+1),
\end{align}
where 
$$
F^{k,\D}_j = f(Y^{k,\D}_{t_j}, \theta_j), \ G^{k,\D}_j = g(Y^{k,\D}_{t_j}, \theta_j).
$$
On the other hand, for $\omega\in\{\beta< j+1\}$, we have $\be\le j$, and hence
\begin{align}\nn
Y^{k,\D}(t_{(j+1)\we\be}) = Y^{k,\D}(t_{j \wedge \beta}) = Y^{k,\D}(t_{j\we \be}) + \big( F^{k,\D}_j\D + G^{k,\D}_j\D B_j \big) \II_{[[0, \be]]}(j+1).
\end{align}
In other words, we always have
\begin{align*}
Y^{k,\D}(t_{(j+1)\we\be}) =Y^{k,\D}(t_{j\we\be}) + \big( F^{k,\D}_j\D + G^{k,\D}_j\D B_j \big) \II_{[[0,\be]]}(j+1).
\end{align*}
Similar to the proof of Theorem \ref{th-BTN}, we obtain that
$$
\big| Y^{k,\D}(t_{(j+1)\we\be}) \big|^2  \le \big| Y^{k,\D}(t_{j\we\be}) \big|^2 + \zeta^{k,\D}_j \II_{[[0,\be]]}(j+1),
$$
here
$$
\begin{aligned}
\zeta^{k,\D}_j &=   \big| F^{k,\D}_j \big|^2 \D^2 + \big| G^{k,\D}_j \D B_j \big|^2 + 2 \big\langle Y^{k,\D}(t_j), F^{k,\D}_j \big\rangle \D   \\ 
& + 2 \big\langle Y^{k,\D}(t_j), G^{k,\D}_j \D B_j \big\rangle + 2 \big\langle F^{k,\D}_j, G^{k,\D}_j \D B_j \big\rangle \D.
\end{aligned}
$$
In a similar way to that of \eqref{b}, we get
\begin{equation*}
\begin{aligned}
& \EE \left( \big|Y^{k,\D}(t_{(j+1)\we\be})\big|^{2p} \big|\CF_{t_{j\we\be}} \right) \\
\le  & \big|Y^{k,\D}(t_{j\we\be})\big|^{2p}  +  \sum_{m=1}^p C_p^m  \EE \left( \left|Y^{k,\D}(t_j)\right|^{2(p-m)}  \left(\zeta^{k,\D}_j\right)^m \II_{[[0,\be]]}(j+1) \Big| \CF_{t_{j\wedge \beta}} \right).
\end{aligned}
\end{equation*}
Since
$$
\D B_j \II_{[[0,\be]]}(j+1) = B(t_{(j+1)\we\be})-B(t_{j\we\be})
$$
and $B(t)$ is a continuous martingale,
by the Doob martingale stopping time theorem (see e.g. \cite[Theorem 1.5]{MY06}) we know that
$\EE\big((\D B_j) \II_{[[0,\be]]}(j+1) |\CF_{t_{j\we\be}}\big)=0$ and
$$
\EE\big(|\D B_j|^2 \II_{[[0,\be]]}(j+1) |\CF_{t_{j\we\be}}\big)\le d \D \II_{[[0,\be]]}(j+1),
$$
where we used the fact that $\mathbf{1}_{[[0,\beta]]}(j+1)$ is $\mathcal{F}_{t_{j\wedge\beta}}$ measurable. Similarly, by applying the property of the conditional expectation, one can also derive 
$$
\begin{aligned}
& \EE\left( |A\D B_j|^{2m} \mathbf{1}_{[[0,\beta]]}(j+1)|\CF_{t_{j\wedge\beta}} \right) \le (2m-1)!!n^m|A|^{2m}\Delta^m \mathbf{1}_{[[0,\beta]]}(j+1),\\
& \mathbb{E}\left( \left(\left\langle z, A\D B_j\right\rangle \right)^{2m}\mathbf{1}_{[[0,\beta]]}(j+1)|\CF_{t_{j\wedge\beta}} \right)=(2m-1)!!|\langle z, A\rangle|^{2m}  \Delta^m \mathbf{1}_{[[0,\beta]]}(j+1),\\
& \EE\left( \left(\left\langle z, A\D B_j\right\rangle \right)^{2m-1}\mathbf{1}_{[[0,\beta]]}(j+1)|\CF_{t_{j\wedge\beta}} \right)=0, \ \forall m\ge 2,\ z \in \mathbb{R}^n \hbox{ and } A \in \mathbb{R}^{n \times d}. \\
\end{aligned}
$$
{ By combining these with the techniques used in the proof of \eqref{bou-final}, and using the fact that $\II_{{[[0, \beta]]}}(j-m+1) \le \II_{{[[0, \beta]]}}(j+1)$ for any $m \in \{-kl, -kl+1, \cdots, -1 \}$, one obtains
	$$
	\EE  \big|Y^{k,\D}(t_{n\we\be})\big|^{2p}
	 \le  CT (1+\|\xi\|_r^{2p} + \|\xi\|_r^{\bar p}) + C\D \sum_{j=0}^{n-1} \EE  \big|Y^{k,\D}(t_{j\we\be})\big|^{2p}.
	$$}
	An application of the discrete Gronwall inequality yields that
	\begin{equation}\label{3.29}
	\sup_{0 \le n \le \lfloor T/\D \rfloor} \EE \big|Y^{k,\D}(t_{n\we\be})\big|^{2p}  
	\le CT (1+\|\xi\|_r^{2p} + \|\xi\|_r^{\bar p}) e^{CT}.
	\end{equation}
	Furthermore, for any $t \in [0, T]$, there exists $j \in \mathbb{N}$ such that $t \in [t_j, t_{j+1})$. Making use of \eqref{ap}, the Burkholder-Davis-Gundy inequality, and \eqref{linear} yields
	\begin{equation*}
	\begin{aligned}
	\EE \big| Z^{k,\D}(t \wedge \varrho) \big|^{2p} & = C \EE \big| X^{k,\D}(t_{j \wedge \beta}) \big|^{2p}  + C  \EE \big| f^{k,\D}_{j \wedge \beta} \big|^{2p} \Delta^{2p} + C \EE \big| g^{k,\D}_{j \wedge \beta} \big|^{2p} \Delta^p \\
	& \le C \bigg( 1 +  \EE  \big|X^{k,\D}(t_{j\we\be})\big|^{2p} + \EE \bigg(\int_{-\infty}^0 \big| X^{k,\D}_{t_{j\wedge \beta}}(u) \big|^2 \mu_1(\d u) \bigg)^p  \bigg) \\
	& \le C \bigg( 1 +  \EE  \big|Y^{k,\D}(t_{j\we\be})\big|^{2p} + \EE \bigg(\int_{-\infty}^0 \big| Y^{k,\D}_{t_{j\wedge \beta}}(u) \big|^2 \mu_1(\d u) \bigg)^p  \bigg).
	\end{aligned}
	\end{equation*}
	By \eqref{3.29} and the approach used in the proof of \cite[Remark 5.5]{LLMS23}, we derive that 
	$$
	\EE \bigg(\int_{-\infty}^0 \big| Y^{k,\D}_{t_{j\wedge \beta}}(u) \big|^2 \mu_1(\d u) \bigg)^p \le C_{T,\xi}.
	$$
	Hence, 
	$$
	\EE \big| Z^{k,\D}(t \wedge \varrho) \big|^{2p} \le C_{T,\xi}.
	$$
	The desired assertion then follows from
	{ $$
		\big( \Gamma^{-1}(\Delta_1^{-1}) \big)^{2p} \mathbb{P}\{\varrho \le T\} \le \mathbb{E} \big(|Z^{k,\D}(T\wedge\varrho)|^{2p}\big)
		\le C_{T,\xi}
		$$}
	and the fact that $C_{T,\xi}$ is independent of $k$ and $\Delta$. The proof is complete.
\end{proof}
$\hfill\square$

{  For each $u \ge 0$, define
\begin{equation}\label{h}
  h(u) = \sup_{v_1, v_2 \in \RR_-: |v_1-v_2| \le u} |\xi(v_1)-\xi(v_2)|.
\end{equation}
As cited by \cite{LMS24, M03, SWMW24, WM08} the uniform continuity of the initial data is important for   the analysis of the strong convergence of numerical schemes for SFDEs. One observes that if $\xi$ is uniformly continuous on $\mathbb{R}_-$, then $h(u) < +\infty$ for any $u \ge 0$, and
\begin{equation}\label{lim-h}
\lim_{u \to 0} h(u) = 0.
\end{equation}}

\begin{lemma}\label{lemma3.4}
Let Assumptions \ref{Lip} and \ref{mon} hold with $p > 1$ and $\xi$ be uniformly continuous on $\RR_-$. Then for any $q\in (0, 2p)$, $\D_1 \in (0, 1]$, $\D \in (0,\D_1]$ and $T > 0$, 
{ $$
 \sup_{k \ge T} \EE \Big( \sup_{0 \le t \le T} |x^k(t)-Z^{k,\D}(t)|^q \II_{\{\s^k_{\D, \D_1} > T \} } \Big)  \le C_{T,\xi,\D_1} \left(  h^q(\D)  + \D^{1\we\frac{q}{4}}  + \frac{1}{\Gamma^{-1}(\Delta^{-\lambda})} \right),
$$
where
$$\sigma^k_{\Delta, \Delta_1} = \tau^k_{\Gamma^{-1}( \Delta_1^{-\lambda})} \wedge  \varrho^k_{\D,\Delta_1},$$
and the stopping times $\tau^k_{\Gamma^{-1}( \Delta_1^{-\lambda})}$, $\varrho^k_{\D,\Delta_1}$ are defined in \eqref{rkh}, \eqref{rh} respectively.}
\end{lemma}
\begin{proof}
By Young's inequality, it is sufficient to consider the case $q \in [2, 2p)$. Fix $T > 0$ and $\D_1 \in (0,1]$ arbitrarily. For any $k \ge T$ and $\D \in (0,\D_1]$, let $\sigma = \sigma^k_{\Delta, \Delta_1}$ for simplicity. It follows from \eqref{ap-con} and the Burkholder-Davis-Gundy inequality that, for any $t \in [0,T]$,
\begin{equation}\label{3.25}
\begin{aligned}
   &  \EE\Big( \sup_{0 \le s \le t} |x^k(s)-Z^{k,\D}(s)|^q \II_{\{\s > t \}} \Big)  \\
   \le & C_T \EE \int_0^{t} \left[ \left( |f_k(x^k_s, \theta(s)) - f(X^{k,\D}_s, \theta_s)|^q + |g_k(x^k_s, \theta(s)) - g(X^{k,\D}_s, \theta_s)|^q \right) \II_{\{\s > t \}} \right] \d s \\
\end{aligned}
\end{equation}
Obviously, using the H$\rm{\ddot o}$lder inequality, an argument similar to that in \cite[(5.15)]{LLMS23}, together with Assumption \ref{Lip} and Corollary \ref{coro3.3}, we obtain
\begin{equation*}
\begin{aligned}
   &  \EE \int_0^{t} \left[ \left( |f_k(x^k_s, \theta(s)) - f(X^{k,\D}_s, \theta_s)|^q \right) \II_{\{\s > t \}} \right] \d s  \\
   \le & 2^{q-1} \EE \int_0^{t}  \left[ \left( |f_k(x^k_s, \theta(s)) - f(X^{k,\D}_s, \theta(s))|^q + |f(X^{k,\D}_s, \theta(s)) - f(X^{k,\D}_s, \theta_s)|^q \right) \II_{\{\s > t \}} \right] \d s \\
   \le & 2^{q-1} \EE \int_0^{t}  \left( |f_k(x^k_s, \theta(s)) - f(X^{k,\D}_s, \theta(s))|^q \II_{\{\s > t \}} \right) \d s + C_{T,\xi,\Delta_1} \Delta. \\
\end{aligned}
\end{equation*}
{ Noting that $\|x^k_s\|_r \vee \|X^{k,\D}_s\|_r \le  \|\xi\|_r  \vee \Gamma^{-1}(\D_1^{-\lambda})$ for any $s \in [0, t] \subset [0, \s)$} and using Assumption \ref{Lip}, we obtain
{ \begin{equation*}
\begin{aligned}
& \EE \int_0^{t} \left( |f_k(x^k_s, \theta(s)) - f(X^{k,\D}_s, \theta(s))|^q \II_{\{\s > t\}} \right) \d s  \\
   \le & C_{\xi,\D_1}  \EE \int_0^{t} \int_{-\infty}^0 \left( |\pi_k(x^k_s)(u) - X^{k,\D}_s(u)|^q \II_{\{\s > t \}} \right) \mu_1(\d u) \d s \\
   = & C_{\xi,\Delta_1} \EE \left[ \left( J_1(t) + J_2(t) \right) \II_{\{\s > t \}} \right],
\end{aligned}
\end{equation*}
where
$$
R_1(t) := \mu_1\left( (-\infty, -k]\right) \int_0^t   |x^k_s(-k)-X^{k,\Delta}_s(-k)|^q   \d s,
$$
$$
R_2(t) := \int_0^t \int_{-k}^0  |x^k_s(u)-X^{k,\Delta}_s(u)|^q \mu_1(\d u) \d s.
$$
}
{ 
For any $t < \s$, let $n = \lfloor t/\Delta \rfloor$ and $t_{n+1} = t$. It follows from $k \ge T$ and the H$\rm{\ddot o}$lder inequality that
\begin{align}\nn
R_1(t) = & \mu_1\left( (-\infty, -k]\right) \sum_{j=0}^{n}\int_{t_j}^{t_{j+1}}  |x^k_s(-k) - X^{k,\Delta}_{t_j}(-k)|^q \d s \\ \nn
 = & \mu_1\left( (-\infty, -k]\right) \sum_{j=0}^{n}\int_{t_j}^{t_{j+1}}  |\xi(s-k) - \Lambda^{\Delta}\left( \xi(t_j-k) \right)|^q \d s \\ \nn
 \le & 2^{q-1} \mu_1\left( (-\infty, -k]\right) \sum_{j=0}^{n}\int_{t_j}^{t_{j+1}}  |\xi(s-k) -  \xi(t_j-k) |^q \d s \\ \nn
 + & 2^{q-1} \Delta \mu_1\left( (-\infty, -k]\right) \sum_{j=0}^{n}   |\xi(t_j-k) - \Lambda^{\Delta}\left( \xi(t_j-k) \right)|^q .
\end{align}
Noting that for any $u \le 0$ and positive integers $\tilde\gamma, \gamma$,
$$
  \left|\xi(u) - \Lambda^{\Delta}(\xi(u)) \right|^{\tilde\gamma} = \left|\xi(u) - \Lambda^{\Delta}(\xi(u)) \right|^{\tilde\gamma} \II_{\{|\xi(u)|>\Gamma^{-1}(\Delta^{-\lambda}) \}} \le \frac{ C |\xi(u)|^{\tilde\gamma+\gamma}}{\left( \Gamma^{-1}(\Delta^{-\lambda}) \right)^\gamma}.
$$
It follows from the uniform continuity of $\xi$ that
\begin{equation}\label{tnt}
 \left|\xi(u) - \Lambda^{\Delta}(\xi(u)) \right|^{\tilde\gamma} \le \frac{C\left(1 + |u|^{\tilde\gamma+\gamma}\right)}{\left(\Gamma^{-1}(\Delta^{-\lambda})\right)^{\gamma}}.
\end{equation}
Since $\mu_1 \in \CP_{2r}$, we have
$$
\begin{aligned}
R_1(t) &  \le C_T h^q(\Delta) + C \Delta \mu_1\left( (-\infty, -k]\right) \sum_{j=0}^{n}   \frac{\left(1 + |t_j-k|^{q+1}\right)}{ \Gamma^{-1}(\Delta^{-\lambda}) } \\
& \le C_T h^q(\Delta) + C_{T,\xi} \mu_1\left( (-\infty, -k]\right)   \frac{\left(1 + k^{q+1}\right)}{ \Gamma^{-1}(\Delta^{-\lambda}) } \\
& \le C_{T} h^q(\Delta) + C_{T,\xi} \mu_1\left( (-\infty, -k]\right) e^{2rk}  \frac{ \left( \sup_{u \le 0} \left(e^{2ru} (1+|u|^{q+1}) \right) \right)   }{ \Gamma^{-1}(\Delta^{-\lambda}) } \\
& \le C_{T} h^q(\Delta) + \frac{C_{T,\xi}}{\Gamma^{-1}(\Delta^{-\lambda})},
\end{aligned}
$$
where $h(\cdot)$ is defined in \eqref{h}.
Moreover, for any $t < \s$, using arguments similar to that in \cite[Lemma 3.3]{M03} and the estimation of $J_1(t)$, we derive that
$$
\begin{aligned}
R_2(t)
\le  C_{T,\xi} \left( h^q(\Delta) + \Delta^{\frac{q}{4}} + \frac{1}{\Gamma^{-1}(\Delta^{-\lambda})} \right)  + \int_0^{t} \int_{-s}^0 |x^k(s+u) - X^{k,\D}_s(u)|^q \mu_1(\d u) \d s.
\end{aligned}
$$
}
Therefore, 
\begin{equation}\label{3.26}
\begin{aligned}
   & \EE \int_0^{t} \left( |f_k(x^k_s, \theta(s)) - f(X^{k,\D}_s, \theta(s))|^q \II_{\{\s > t\}} \right) \d s \\
   \le & C_{T,\xi,\D_1} \left( h^q(\D) +  \D^{\frac{q}{4}}   + \frac{1}{\Gamma^{-1}(\Delta^{-\lambda})}  \right)   +   \EE \int_0^{t} \int_{-s}^0 \left( |x^k(s+u) - X^{k,\D}_s(u)|^q \mu_1(\d u) \II_{\{\s > t\}} \right) \d s.
\end{aligned}
\end{equation}
By the triangle inequality and the standard  argument (see e.g. \cite[(2.14)]{LLMS23}),
\begin{align}\nn
   &  \EE \int_0^{t} \int_{-s}^0 \left( |x^k(s+u) - X^{k,\D}_s(u)|^q \II_{\{\s > t \}}  \right) \mu_1(\d u) \d s \\ \nn
   \le & 2^{q-1} \EE \int_0^{t} \int_{-s}^0  \left( |x^k(s+u) - Z^{k,\D}(s+u)|^q \II_{\{\s > t \}}  \right)  \mu_1(\d u) \d s \\ \nn
    & + 2^{q-1} \EE \int_0^{t} \int_{-s}^0 \left( |Z^{k,\D}(s+u) - X^{k,\D}_s(u)|^q \II_{\{\s > t \}}\right)  \mu_1(\d u) \d s \\ \nn
   \le & 2^{q-1}\EE \int_0^{t}  \left( |x^k(s) - Z^{k,\D}(s)|^q \II_{\{\s > t \}} \right) \d s \\ \nn
   & + 2^{q-1} \EE \int_0^{t} \int_{-s}^0 \left( |Z^{k,\D}(s+u) - X^{k,\D}_s(u)|^q \II_{\{\s > t \}} \right)  \mu_1(\d u) \d s. 
\end{align}
An argument similar to that in \cite[Lemma 3.3]{M03}, combined with \eqref{linear}, leads to
\begin{equation*}
\begin{aligned}
   \EE \int_0^{t} \int_{-s}^0 \left(  |Z^{k,\D}(s+u) - X^{k,\D}_s(u)|^q \II_{\{\s > t \}} \right) \mu_1(\d u) \d s \le C_{T,\xi} \big( h^q(\D) + \D^{\frac{q}{4}} \big).
\end{aligned}
\end{equation*}
Hence, we have
\begin{equation*}
\begin{aligned}
   & \EE \int_0^{t} \int_{-s}^0 \left(  |x^k(s+u) - X^{k,\D}_s(u)|^q \II_{\{\s > t \}} \right) \mu_1(\d u) \d s \\
   \le &  C \EE \int_0^{t} \left(  |x^k(s) - Z^{k,\D}(s)|^q \II_{\{\s > t \}} \right) \d s + C_{T,\xi} \big( h^q(\D) + \D^{\frac{q}{4}} \big). \\
\end{aligned}
\end{equation*}
Inserting the above inequality into \eqref{3.26} gives
\begin{equation}\label{3-f}
\begin{aligned}
   & \EE \int_0^{t} \left( |f_k(x^k_s, \theta(s)) - f(X^{k,\D}_s, \theta(s))|^q \II_{\{\s > t \}} \right) \d s \\
   \le & C_{T,\xi,\D_1} \left(  h^q(\D)  + \D^{1\we\frac{q}{4}}  + \frac{1}{\Gamma^{-1}(\Delta^{-\lambda})} \right)   +  C \EE \int_0^{t} \left( |x^k(s) - Z^{k,\D}(s)|^q \II_{\{\s > t \}} \right) \d s \\
      \le & C_{T,\xi,\D_1} \left(  h^q(\D)  + \D^{1\we\frac{q}{4}} + \frac{1}{\Gamma^{-1}(\Delta^{-\lambda})} \right)   +   C \int_0^{t}  \EE\Big( \sup_{0 \le u \le s} |x^k(u) - Z^{k,\D}(u)|^q \II_{\{\s > t \}} \Big) \d s. \\
\end{aligned}
\end{equation}
Similarly, 
\begin{equation}\label{3-g}
\begin{aligned}
   & \EE \int_0^{t\we\s} |g_k(x^k_s, \theta(s)) - g(X^{k,\D}_s, \theta(s))|^q \d s \\
      \le & C_{T,\xi,\D_1} \left(  h^q(\D)  + \D^{1\we\frac{q}{4}}  + \frac{1}{\Gamma^{-1}(\Delta^{-\lambda})} \right)   +   C \int_0^{t}  \EE\Big( \sup_{0 \le u \le s} |x^k(u) - Z^{k,\D}(u)|^q \II_{\{\s > t \}} \Big) \d s. \\
\end{aligned}
\end{equation}
Substituting \eqref{3-f} and \eqref{3-g} into \eqref{3.25} yields
\begin{equation*}
\begin{aligned}
   &  \EE\Big( \sup_{0 \le s \le t} |x^k(s)-Z^{k,\D}(s)|^q \II_{\{\s > t \}} \Big)  \\
   \le & C_{T,\xi,\D_1} \left(  h^q(\D)  + \D^{1\we\frac{q}{4}}  + \frac{1}{\Gamma^{-1}(\Delta^{-\lambda})} \right)   +   C_T \int_0^{t}  \EE\Big( \sup_{0 \le u \le s} |x^k(u) - Z^{k,\D}(u)|^q \II_{\{\s > t \}} \Big) \d s \\
   \le & C_{T,\xi,\D_1} \left(  h^q(\D)  + \D^{1\we\frac{q}{4}}  + \frac{1}{\Gamma^{-1}(\Delta^{-\lambda})} \right)   +   C_T \int_0^{t}  \EE\Big( \sup_{0 \le u \le s} |x^k(u) - Z^{k,\D}(u)|^q \II_{\{\s > s \}} \Big) \d s. \\
\end{aligned}
\end{equation*}
By the Gronwall inequality, we obtain
\begin{equation*}
\begin{aligned}
   \EE\Big( \sup_{0 \le s \le t} |x^k(s)-Z^{k,\D}(s)|^q \II_{\{\s > t \}} \Big)  
   \le  C_{T,\xi,\D_1} \left(  h^q(\D)  + \D^{1\we\frac{q}{4}}  + \frac{1}{\Gamma^{-1}(\Delta^{-\lambda})} \right) . 
\end{aligned}
\end{equation*}
The desired assertion then follows from the fact that $C_{T,\xi,\Delta_1}$ is independent of $k$. The proof is complete.
\end{proof}
$\hfill\square$

\begin{lemma}\label{l3.6}
Let the assumptions in Lemma \ref{lemma3.4} hold. Then for any $q\in (0, 2p)$,
$$
\lim_{\D \to 0} \sup_{k \ge T} \sup_{0 \le t \le T}\EE |x^k(t)-X^{k,\D}(t)|^q =0, \quad \forall T > 0.
$$
\end{lemma}
\begin{proof}
The Lyapunov inequality implies that it is enough to consider the case $q \in [2,2p)$. Let $T > 0$ and $\Delta_1 \in (0, 1]$ be arbitrary. For any $k \ge T$ and $\D \in (0,\D_1]$, it follows from Young's inequality that for any $\delta > 0$ and $t \in [0,T]$, 
\begin{equation*}
\begin{aligned}
  & \EE \big|x^k(t) - X^{k,\D}(t)\big|^q \\ 
= & \EE\left( \big|x^k(t)-X^{k,\D}(t)\big|^q\II_{\{\s\le T\}} \right) 
+ \EE\left( \big|x^k(t)-X^{k,\D}(t)\big|^q\II_{\{\s > T\}} \right) \\
\le & \frac{q\de}{2p}\EE\big|x^k(t) - X^{k,\D}(t)\big|^{2p} + \frac{2p-q}{2p\de^{q/(2p-q)}}\PP(\s\le T) \\
+ & 2^{q-1} \EE\left( \big|x^k(t)-Z^{k,\D}(t)\big|^q\II_{\{\s >T\}} \right) + 2^{q-1} \EE\left( \big|Z^{k,\D}(t)-X^{k,\D}(t)\big|^q\II_{\{\s >T\}} \right).
\end{aligned}
\end{equation*}
By virtue of Lemma \ref{lemma3.1} and Theorem \ref{th-BTN}, 
\begin{equation*}
\begin{aligned}
  \frac{q\de}{2p}\EE\big|x^k(t) - X^{k,\D}(t)\big|^{2p}
\le  \frac{2^{2(p-1)}q\de}{p}\left( \EE\big|x^k(t)\big|^{2p} + \EE\big|X^{k,\D}(t)\big|^{2p} \right)  \le C_{T,\xi} \delta.
\end{aligned}
\end{equation*}
It follows from $\eqref{BTSFDE}$ and Lemma \ref{l4.6} that
$$
\begin{aligned}
   \frac{2p-q}{2p\de^{q/(2p-q)}}\PP(\s\le T)  
  & \le  \frac{2p-q}{2p\de^{q/(2p-q)}} \big( \PP(\tau^k_{\G^{-1}(\D_1^{-\lbd})}\le T) + \PP(\varrho^k_{\D, \D_1}\le T) \big) \\
  & \le  \frac{ C_{T,\xi}}{ \de^{q/(2p-q)} \big( \G^{-1}(\D_1^{-\lbd}) \big)^{2p}}.
\end{aligned}
$$
Moreover, for any $t \in [0,T]$, there exists a unique $j \in \mathbb{N}$ such that $t \in [t_j, t_{j+1})$. Then,  it is easy to verify from the Burkholder-Davis-Gundy inequality, \eqref{linear}, Corollary \ref{coro3.3} and $\lambda \in (0, 1/2]$ that
\begin{equation*}
    \EE\left( \big|Z^{k,\D}(t)-X^{k,\D}(t)\big|^q \II_{\{\s >T\}} \right)
   \le  \EE \left( |f^{k,\D}_j|^q \D^q + |g^{k,\D}_j|^q \D^{\frac{q}{2}}  \right)
   \le  C_{T, \xi} \Delta^{\frac{q}{4}}.
\end{equation*}
Hence, 
\begin{equation}\label{3.38}
\begin{aligned}
\EE \big|x^k(t) - X^{k,\D}(t)\big|^q 
\le & C_{T,\xi} \delta  +   C_{T,\xi}  \de^{q/(q-2p)} \big( \G^{-1}(L\D_1^{-\lbd}) \big)^{-2p}  + C_{T,\xi} \Delta^{\frac{q}{4}} \\
+ & 2^{q-1} \EE \left( \sup_{0 \le t \le T} \big|x^k(t)-Z^{k,\D}(t)\big|^q \II_{\{\s > T \}}\right). \\
\end{aligned}
\end{equation}
An application of Lemma \ref{lemma3.4} yields that for any $k \ge T$, 
$$
\begin{aligned}
\EE \big|x^k(t) - X^{k,\D}(t)\big|^q  
   \le &    C_{T,\xi} \delta  +   C_{T,\xi}  \de^{q/(q-2p)} \big( \G^{-1}(\D_1^{-\lbd}) \big)^{-2p} \\
  + &  C_{T,\Delta_1,\xi} \left( h^q(\Delta)  + \Delta^{1 \wedge \frac{q}{4}} + \big( \G^{-1}(\D^{-\lbd}) \big)^{-1}  \right).
\end{aligned}
$$
{  We derive the desired assertion by the standard argument (see e.g. \cite[Theorem $3.2$]{LMS24} or \cite[Theorem $3.3$]{SHGL22}).}
The proof is complete.
\end{proof}
$\hfill\square$

Lemma \ref{l3.6} reveals that $X^{k,\Delta}(t)$ converges to $x^k(t)$ as $\Delta \to 0$ while Lemma \ref{lemma3.1} establishes the convergence between $x^k(t)$ and $x(t)$ as $k \to +\infty$. Combining both them, we  conclude that $X^{k,\Delta}(t)$ converges to $x(t)$ as $k \to +\infty$ and $\Delta \to 0$.

\begin{theorem}\label{th3.7}
Let Assumptions \ref{Lip} and \ref{mon} hold with $p > 1$, $\mu_1 \in \CP_{b}$ with $b > 2r$, and $\xi$ be uniformly continuous on $\RR_-$. Then, for any $q\in (0, 2p)$,
\begin{equation}\label{3.41}
\lim_{\D\rightarrow 0, k \to +\infty} \sup_{0 \le t \le T}\EE|x(t)-X^{k,\D}(t)|^q=0, \quad \forall T>0.
\end{equation}
\end{theorem}

Theorem \ref{th3.7} shows that the order of the limit processes $\Delta \to 0$ and $k \to +\infty$ can be interchanged. This follows from the uniform convergence of $X^{k,\Delta}(t)$ to $x^k(t)$ with respect to $k$, as established in Lemma \ref{l3.6}.

\section{Convergence rate}\label{S4}

The convergence of the TEM solution $X^{k,\Delta}(t)$ to the exact $x(t)$ is established in Theorem \ref{th3.7}. This section goes a further step to establish the convergence rate of the TEM numerical solution by choosing the appropriate $k$ and $\Delta$.
In order for the convergence accuracy we impose a slightly stronger assumption. For any $\iota_1, \iota_2>0$, let $\CU_{\iota_1,\iota_2}$ denote the family of continuous functions $U: \RR^n \K \RR^n \rightarrow \RR_+$ satisfying $U(x,x) = 0$ and
$$
\sup_{  x,y\in\RR^n, x \neq y}\frac{U(x,y)}{|x-y|^{\iota_1}(1+|x|^{\iota_2}+|y|^{\iota_2})}<\infty.
$$

\begin{assp}\label{a4.1}
There exist constants $d_1>0$ and $d_2 \ge 1/2$ such that the initial data $\xi$ satisfies
$$
|\xi(t_1)-\xi(t_2)|\le d_1|t_1-t_2|^{d_2},\quad \forall t_1, t_2 \in (-\8, 0].
$$
\end{assp}

{ \begin{assp}\label{a4.2}
There exist constants $\td p>2$,  $ d_3, \iota>0$ and $\a_1>2r$, $\a_2>(2+\iota) r$,   a function $U \in \CU_{2,\iota}$, and measures $\n_1\in \CP_{\a_1}, \n_2\in\CP_{\a_2}$ such that  
$$
\begin{aligned}
& 2 \big\langle \p(0)-\f(0), f(\p,i)-f(\f,i)\big\rangle + (\td p-1)|g(\p,i)-g(\f,i)|^2 \\
\le &d_3  \int_{-\8}^0|\p(u)-\f(u)|^2\n_1(\d u) 
- U(\p(0),\f(0))+\int_{-\8}^0 U(\p(u),\f(u))\n_2(\d u)
\end{aligned}
$$
for any $\p,\f\in \CC_r$ and   $i\in\SM$.
\end{assp}}

\begin{assp}\label{a4.3}
There exist constants $d_4, v>0$ and  measures   $\nu_3\in\CP_{2r}$, $\n_4, \n_5 \in \CP_{vr}$ such that  for any $\p, \f\in \CC_r$ and  $i\in\SM$
$$
\begin{aligned}
|f(\p,i)-f(\f,i)|
\le  d_4\int_{-\8}^0|\p(u)-\f(u)|\n_3(\d u) \left( 1+ \int_{-\8}^0\big(|\p(u)|^v+|\f(u)|^v\big)\n_4(\d u)   \right),
\end{aligned}
$$
$$
\begin{aligned}
|g(\p,i)-g(\f,i)|^2
\le  d_4\int_{-\8}^0|\p(u)-\f(u)|^2\n_3(\d u)
\left( 1+  \int_{-\8}^0\big(|\p(u)|^v+|\f(u)|^v\big)\n_5(\d u)   \right).
\end{aligned}
$$
\end{assp}
{\begin{rmk}\label{rmk4.4}
Under Assumption \ref{a4.3}, choose~
$ 
\G(R)=  d_5 \left( 1 + d_6 R^v  \right), \quad \forall R \ge 0 
$  in \eqref{G},
where
$ 
d_5 = 2 \left( d_4 \vee \left( \sup_{i\in\SM}|f(0,i)|\right) \vee \left( \sup_{i\in\SM}|g(0,i)|^2 \right) \right), ~~d_6 =\left( \nu_4^{(vr)} \vee  \nu_5^{(vr)}\right).
$ 
Then the truncation mapping   
$ 
\Lambda^{\Delta}(x) = \left(|x| \wedge \left(\frac{ \Delta^{-\lbd }}{d_5d_6}- \frac{1}{d_6} \right)^{\frac{1}{v}}  \right) \frac{x}{|x|}, \quad \forall x \in \RR^n,
$ 
where 
$ 
\lbd = \frac{v}{2(p-1)}. 
$ Obviously, $0 <\lbd \le1/2$ if $p \ge v+1$. Thus,  $f$ and $g$ satisfy
{ \begin{equation}\label{linear-rate}
\begin{aligned}
|f(X^{k,\D}_{t_j},i)|^2 &\le  \D^{-2\lbd}\left( 1 + \int_{-\8}^0 |X^{k,\D}_{t_j}(u)|^2  \nu_3(\d u) \right),\\
|g(X^{k,\D}_{t_j},i)|^2  &\le \D^{-\lbd} \left( 1 + \int_{-\8}^0 |X^{k,\D}_{t_j}(u)|^2 \nu_3(\d u) \right).
\end{aligned}
\end{equation}}
\end{rmk}

We cite the result on the convergence   of $x^k(t)$ to $x(t)$ from \cite[Theorem 3.7]{LLMS23}.
\begin{lemma}\label{th4.5} {\bf(\cite[Theorem 3.7]{LLMS23})}
Let Assumptions \ref{a4.2} and \ref{a4.3} hold. Then both HSFDE \eqref{ISFDE} and \eqref{TSFDE} has a unique global solution  $x(t)$ and $x^k(t)$, respectively,  for $t\in (-\8, \8)$ with the property
\begin{equation*}
\sup_{0 \le t \le T}\EE|x(t)-x^k(t)|^2\le C_T e^{-\alpha k}, \quad \forall T > 0, ~ k \ge T.
\end{equation*}
where $\alpha = \left(\alpha_1 - 2r \right) \wedge \left( \alpha_2 - (2+\iota)r \right)$.
\end{lemma}

By virtue of Lemma \ref{th4.5}, we turn to 
analyze the convergence rate of $X^{k,\D}(t)$ to $x^k(t)$.
For any $t \in [0,T]$, let
\begin{equation}\label{def-Z_t}
    Z^{k,\D}_t(u)=\left\{
\begin{aligned}
    & Z^{k,\D}(t+u),\quad u \in [-k,0],\\
    & Z^{k,\D}(t-k), \quad u \in (-\infty, -k),
\end{aligned}\right.
\end{equation}
where $Z^{k,\Delta}(\cdot)$ is defined by \eqref{ap}.
The error between   $Z^{k,\Delta}_t(\cdot)$ and  $X^{k,\Delta}_t(\cdot)$ is estimated as follows, where $X^{k,\Delta}_{t}(\cdot)$ is defined in \eqref{n-seg}.

{ \begin{lemma}\label{l4.7}
Let Assumptions \ref{mon}, \ref{a4.1} and \ref{a4.3} hold with $p \ge v+1$. Then for any $\mu \in \CP_{2r}$, $T > 0$ and $\D \in (0,1]$,
$$
   \sup_{k \ge 1} \sup_{0 \le t \le T}   \int_{-\infty}^0 \EE  \big|Z^{k,\D}_t(u)-X^{k,\D}_t(u)\big|^{\frac{4p}{v+2}}  \mu(\d u)  \le C_{T,\xi} \D^{\frac{2p}{v+2}}.
$$
\end{lemma}}

\begin{proof}
{ Let $k \ge 1$, $T > 0$ and $\D \in (0, 1]$ be given arbitrarily. For any $t \in [0,T]$ and $u\in[-k,0)$, there exists a unique pair of integers $j \ge 0$ and $m\le -1$ such that $t\in[t_j, t_{j+1})$ and $u\in[t_m, t_{m+1})$. Clearly,
$$
  t_j+u\in[t_{j+m},t_{j+m+1}), \quad t+u\in[t_{j+m},t_{j+m+2}).
$$
To obtain the desired result, we divide into five cases to discuss.

\textbf{Case $\bm{1.}$ $\bm{t+u\in [t_{j+m},t_{j+m+1})\subset(-\infty,0).}$}It follows from \eqref{LI} and \eqref{ap} that
$$
\begin{aligned}
&|Z^{k,\D}_t(u) - X^{k,\D}_t(u)|^{\frac{4p}{v+2}} = |Z^{k,\Delta}(t+u) - X^{k,\Delta}_{t_j}(u)|^{\frac{4p}{v+2}} \\
=&\Big| \xi(t+u)-\frac{t_{m+1}-u}{\D} \Lambda^{\Delta}\left(\xi(t_{j+m}) \right)-\frac{u-t_{m}}{\D} \Lambda^{\Delta} \left( \xi(t_{j+m+1}) \right) \Big|^{\frac{4p}{v+2}} \\
\le & 4^{\frac{4p-v-2}{v+2}} \Big(| \xi(t+u) - \xi(t_{j+m}) |^{\frac{4p}{v+2}} +  | \xi(t+u) - \xi(t_{j+m+1}) |^{\frac{4p}{v+2}} \\
+ &  | \xi(t_{j+m}) - \Lambda^{\Delta}(\xi(t_{j+m})) |^{\frac{4p}{v+2}} +  | \xi(t_{j+m+1}) - \Lambda^{\Delta}(\xi(t_{j+m+1})) |^{\frac{4p}{v+2}} \Big) \\
\end{aligned}
$$
By Assumption \ref{a4.1}, Jensen's inequality and $d_2 \ge 1/2$, 
$$
 |\xi(t+u)-\xi(t_{j+m})|^{\frac{4p}{v+2}} \vee |\xi(t+u)-\xi(t_{j+m+1})|^{\frac{4p}{v+2}} \le d_1^{\frac{4p}{v+2}} \D^{\frac{2p}{v+2}}.
$$
Applying \eqref{tnt} with $\tilde\gamma = 4p/(v+2)$ and $\gamma = 4p(p-1)/(v+2)$, together with Remark \ref{rmk4.4}, it follows that
$$
\left|\xi(u)-\Lambda^{\Delta}(\xi(u))\right                                  |^{\frac{4p}{v+2}} \le C \left( 1 + |u|^{\frac{4p^2}{v+2}} \right) \Delta^{\frac{2p}{v+2}}, \quad \forall u \le 0.
$$
Therefore, 
$$
\begin{aligned}
\EE \big|Z^{k,\D}_t(u) - X^{k,\D}_t(u)\big|^{\frac{4p}{v+2}} 
\le & C \left( 1 + |t_{j+m}|^{\frac{4p^2}{v+2}} + |t_{j+m+1}|^{\frac{4p^2}{v+2}} \right) \Delta^{\frac{2p}{v+2}} \\
\le & C_T \left( 1 + |u|^{\frac{4p^2}{v+2}} \right) \Delta^{\frac{2p}{v+2}}. 
\end{aligned}
$$

\textbf{Case $\bm{2.}$ $\bm{t+u\in[t_{j+m+1},t_{j+m+2})\subset(-\infty,0).}$}
Utilizing \eqref{LI} and \eqref{ap} again yields
$$
\begin{aligned}
& |Z^{k,\D}_t(u) - X^{k,\D}_t(u)| \\
= & \Big| \xi(t+u)  -\frac{t_{m+1}-u}{\D} \Lambda^{\Delta}(\xi(t_{j+m}))-\frac{u-t_m}{\D} \Lambda^{\Delta}(\xi(t_{j+m+1})) \Big| \\
\le & |\xi(t+u)-\xi(t_{j+m+1})| +  |\xi(t_{j+m+1})-\Lambda^{\Delta}(\xi(t_{j+m+1}))| + | \Lambda^{\Delta}(\xi(t_{j+m+1})) -  \Lambda^{\Delta}(\xi(t_{j+m}))|\\
\le & |\xi(t+u)-\xi(t_{j+m+1})| +  |\xi(t_{j+m+1})-\Lambda^{\Delta}(\xi(t_{j+m+1}))| + |\xi(t_{j+m+1}) -\xi(t_{j+m})|.\\
\end{aligned}
$$
Then, as in \textbf{Case $\bm{1}$}, we deduce that
$$
\EE\big|Z^{k,\D}_t(u) - X^{k,\D}_t(u)\big|^{\frac{4p}{v+2}}\le C_T \left( 1 + |u|^{\frac{4p^2}{v+2}} \right) \Delta^{\frac{2p}{v+2}}.
$$}

\textbf{Case $\bm{3.}$ $\bm{t+u\in[t_{j+m},t_{j+m+1})\subset[0,\8).}$} 
In this case, $j+m\ge 0$. By \eqref{LI} and \eqref{ap},
$$
\begin{aligned}
& \big| Z^{k,\D}_t(u) - X^{k,\D}_t(u) \big|  \\
\le & \big| X^{k,\D}(t_{j+m+1}) - X^{k,\D}(t_{j+m}) \big| +  \big| f^{k,\D}_{j+m}\big| \D  + \big| g^{k,\D}_{j+m}(B(t+u)-B(t_{j+m})) \big|   \\
\le & \big|Y^{k,\D}(t_{j+m+1})-X^{k,\D}(t_{j+m})\big| + \big| f^{k,\D}_{j+m}\big| \D  + \big| g^{k,\D}_{j+m}(B(t+u)-B(t_{j+m})) \big| \\
\le & 2\big|f^{k,\D}_{j+m}\big|\D  + \big|g^{k,\D}_{j+m}\big|\big|\D B_{j+m}\big|  + \big|g^{k,\D}_{j+m}\big|\big|B(t+u)-B(t_{j+m})\big|,
\end{aligned}
$$
where $f^{k,\D}_{j+m}$ and $g^{k,\D}_{j+m}$ are defined in the proof of Theorem \ref{th-BTN}.
Then,
\begin{equation}\label{4.2}
\begin{aligned}
      \EE  \big|Z^{k,\D}_t(u)-X^{k,\D}_t(u)\big|^{\frac{4p}{v+2}}  \le  C \EE\left(\big|f^{k,\D}_{j+m}\big|^{\frac{4p}{v+2}}\D^{\frac{4p}{v+2}} + \big|g^{k,\D}_{j+m}\big|^{\frac{4p}{v+2}}\D^{\frac{2p}{v+2}}\right).
\end{aligned}
\end{equation}
It follows from \eqref{linear-rate}, the fact that $\lbd \in (0, 1/2]$ and Corollary \ref{coro3.3} that
\begin{equation}\label{est-f}
\begin{aligned}
      \EE \left( \big| f^{k,\D}_{j+m} \big|^{\frac{4p}{v+2}} \D^{\frac{4p}{v+2}}\right)  \le  C \D^{\frac{2p}{v+2}}  \left(1 + \EE  \left( \int_{-\8}^0 \big|X^{k,\D}_{t_{j+m}}(u)\big|^2 \nu_3(\d u)\right)^{\frac{2p}{v+2}} \right) \le C_{T,\xi} \D^{\frac{2p}{v+2}}.
\end{aligned}
\end{equation}
While, by virtue of Assumption \ref{a4.3} and Young's inequality,
\begin{equation*}
\begin{aligned}
      \EE\big|g^{k,\D}_{j+m}\big|^{\frac{4p}{v+2}}
      \le& C + C\EE\bigg[ \int_{-\8}^0\big|X^{k,\D}_{t_{j+m}}(u)\big|^2 \n_3(\d u) \bigg( 1 + \int_{-\8}^0 \big|X^{k,\D}_{t_{j+m}}(u)\big|^v \n_5(\d u) \bigg) \bigg] ^{\frac{2p}{v+2}} \\
      \le & C  +   C \EE \left( \int_{-\8}^0 \big| X^{k,\D}_{t_{j+m}}(u) \big|^2 \n_3(\d u) \right)^p  +   C  \EE \left( \int_{-\8}^0 \big| X^{k,\D}_{t_{j+m}}(u) \big|^v \n_5(\d u) \right)^{\frac{2p}{v}}.
\end{aligned}
\end{equation*}
It follows from $\nu_5 \in \CP_{vr}$, $p \ge v+1$ and Corollary \ref{coro3.3} that 
\begin{equation}\label{est-g}
\EE \big|g^{k,\D}_{j+m}\big|^{\frac{4p}{v+2}} \le C_{T,\xi}.
\end{equation}
Inserting \eqref{est-f} and \eqref{est-g} into \eqref{4.2}, we obtain
$$
\EE \big|Z^{k,\D}_t(u) - X^{k,\D}_t(u)\big|^{\frac{4p}{v+2}} \le C_{T,\xi} \D^{\frac{2p}{v+2}}.
$$

\textbf{Case $\bm{4.}$ $\bm{t+u\in[t_{j+m+1},t_{j+m+2})=[0,\D).}$}
In this case, we have
$$t_j+u\in [t_{j+m},t_{j+m+1})=[-\D, 0).$$
Thus, 
{ \begin{equation*}
\begin{aligned}
\big|Z^{k,\D}_t(u) - X^{k,\D}_t(u)\big| 
\le & |\Lambda^{\Delta}(\xi(0))- \Lambda^{\Delta}(\xi(-\D))|  +  \big| f^{k,\D}_0 (t+u) \big|   +  \big| g^{k,\D}_0 B(t+u) \big| \\
\le & |\xi(0)- \xi(-\D)|  +  \big| f^{k,\D}_0\big| \Delta   +  \big| g^{k,\D}_0 \big| \big| B(t+u) \big|. \\
\end{aligned}
\end{equation*}}
Making use of Assumption \ref{a4.3}, we arrive at
$$
|f^{k,\D}_0| \vee  |g^{k,\D}_0| \le C \left( 1 + \|\xi\|_r^{v+1} \right).
$$
This, together with Assumption \ref{a4.1}, $d_2 \ge 1/2$ and $t+u \in (0,\D)$, results in 
$$
\EE \big|Z^{k,\D}_t(u) - X^{k,\D}_t(u) \big|^{\frac{4p}{v+2}} \le C_{\xi} \D^{\frac{2p}{v+2}}.
$$

\textbf{Case $\bm{5.}$ $\bm{t+u\in [t_{j+m+1}, t_{j+m+2})\subset[\D,\8).}$} Clearly, $t_j+u\in [0, \8)$. Then, by the same argument as in \textbf{Case $\bm3$}, the desired assertion follows.

In summary, we can conclude that
$$
\EE \big| Z^{k,\D}_t(u) - X^{k,\D}_t(u) \big|^{\frac{4p}{v+2}} \le C_{T,\xi} \left( 1 + |u|^{\frac{4p^2}{v+2}} \right) \Delta^{\frac{2p}{v+2}}, \quad  u \in [-k, 0).
$$
When $u=0$ we can derive that
\begin{equation}\label{ZX0}
\EE \big| Z^{k,\D}_t(0) - X^{k,\D}_t(0) \big|^{\frac{4p}{v+2}} \le C_{T,\xi}  \D^{\frac{2p}{v+2}}
\end{equation}
by the same argument as in \textbf{Case $\bm3$}.
Hence,
$$
\EE \big| Z_t^{k,\D}(u) -  X_t^{k,\D}(u)\big|^{\frac{4p}{v+2}} \le C_{T,\xi} \left( 1 + |u|^{\frac{4p^2}{v+2}} \right) \Delta^{\frac{2p}{v+2}}, \quad \forall u \in [-k,0].
$$
{ Moreover, by \eqref{LI} and \eqref{def-Z_t}, for any $u < -k$,
$$
\EE \big| Z_t^{k,\D}(u) -  X_t^{k,\D}(u)\big|^{\frac{4p}{v+2}} = \EE \big| Z_t^{k,\D}(-k) -  X_t^{k,\D}(-k)\big|^{\frac{4p}{v+2}} \le C_{T,\xi} \left( 1 + |u|^{\frac{4p^2}{v+2}} \right) \Delta^{\frac{2p}{v+2}}.
$$
Therefore, it follows from $\mu \in \CP_{2r}$ that
$$
\begin{aligned}
\int_{-\infty}^0 \EE \big|Z^{k,\Delta}_t(u) - X^{k,\Delta}_t(u) \big|^{\frac{4p}{v+2}} \mu(\d u)
 \le & C_{T,\xi} \Delta^{\frac{2p}{v+2}}  \int_{-\infty}^0 \left( 1 + |u|^{\frac{4p^2}{v+2}} \right) \mu(\d u) \\
 \le &  C_{T,\xi} \mu^{(2r)}  \Delta^{\frac{2p}{v+2}} \left[\sup_{u \le 0} \left( e^{2ru} \left( 1 +  |u|^{\frac{4p^2}{v+2}} \right) \right) \right]    \\
 \le & C_{T,\xi} \Delta^{\frac{2p}{v+2}}.
\end{aligned}
$$
The proof is complete.}
\end{proof}
$\hfill\square$

Based on Lemma \ref{l4.7}, we obtain the convergence rate   between the TEM numerical solution and the exact  of the SFDE \eqref{TSFDE}.
\begin{theorem}\label{l4.9}
Let Assumptions \ref{mon}, \ref{a4.1}-\ref{a4.3} hold with $p \ge  1+3v/2 $. Then 
$$
      \sup_{k \ge 1} \sup_{0 \le t \le T} \EE\big|x^k(t)-X^{k,\D}(t)\big|^2 \le C_{T,\xi}  \D, \quad \forall T>0,\ \Delta \in (0,1].
$$
\end{theorem}

\begin{proof}
Fix $k \ge 1$, $T > 0$ and $\D \in (0,1]$ arbitrarily.
Let 
$$e^{k,\D}(t) = x^k(t)-Z^{k,\D}(t),~\forall t \le T, \quad \vo^k_{\D}:=\tau^k_{\G^{-1}(L\D^{-\lbd})} \we \varrho^k_{\Delta,\D},$$
where $\tau^k_{\G^{-1}(L\D^{-\lbd})}$, $\varrho^k_{\Delta,\D}$ are defined by \eqref{rkh} and \eqref{rh} respectively. Obviously, for any $t \in [0, T]$,
\begin{equation*}
\begin{aligned}
 & \EE \big| x^k(t) - X^{k,\D}(t) \big|^2 \\
= & \EE\big(\big| x^k(t) - X^{k,\D}(t) \big|^2  \II_{\{\vo^k_\D \le T\}}\big)  +  \EE \big(\big| x^k(t) - X^{k,\D}(t) \big|^2 \II_{\{\vo^k_\D >  T\}}\big) \\
\le & \EE\left(\big| x^k(t) - X^{k,\D}(t) \big|^2  \II_{\{\vo^k_\D \le T\}}\right) + \EE \left(\big|e^{k,\D}(t)\big|^2  \II_{\{\vo^k_\D >  T\}} \right)   +  \EE  \big|Z^{k,\D}(t) - X^{k,\D}(t)\big|^2. \\
\end{aligned}
\end{equation*}
Using Lemmas \ref{lemma3.1} and \ref{l4.6}, Remark \ref{rmk4.4}, and following a same argument as in the proof of \cite[Lemma 3.7]{LMS24}, we obtain
$$
      \sup_{k \ge 1} \sup_{0 \le t \le T} \EE\big(\big| x^k(t) - X^{k,\D}(t) \big|^2  \II_{\{\vo^k_\D \le T\}}\big)   \le C_{T,\xi} \D.
$$
In addition, for any $t \in [0, T]$, there exists a unique nonnegative integer $j$ such that $t \in [t_j, t_{j+1})$. { It follows from $p \ge (3v/2) +1$, the H$\rm{\ddot{o}}$lder inequality and \eqref{ZX0} that
$$
   \EE \big|Z^{k,\D}(t) - X^{k,\D}(t) \big|^2 \le  \big( \EE \big| Z^{k,\D}_t(0) - X^{k,\D}_t(0) \big|^{\frac{4p}{v+2}} \big)^{\frac{v+2}{2p}} \le C_{T,\xi} \D.
$$}
Hence,
\begin{equation}\label{4-rate}
\begin{aligned}
 \EE \big| x^k(t) - X^{k,\D}(t) \big|^2 
\le &  C_{T,\xi} \D + \EE \big(\big|e^{k,\D}(t)\big|^2  \II_{\{\vo^k_\D >  T\}} \big).
\end{aligned}
\end{equation}
By \eqref{TSFDE} and \eqref{ap-con}, for any $t \in [0, T] \subset [0, \vo^k_\D)$,
$$
\begin{aligned}
 e^{k,\D}(t) 
=& \int_0^{t} \left( f_k(x^k_s,\o(s))-f(X_s^{k,\D},\o_s) \right) \d s 
+ \int_0^{t} \left( g_k(x^k_s,\o(s))-g(X_s^{k,\D},\o_s) \right) \d B(s).
\end{aligned}
$$
Using the generalized It$\hat{\text{o}}$ formula (\cite[Lemma $1.9$]{MY06}) 
and the inequality $(x+y)^2 \le (1+\varepsilon) x^2 + (1 + (1/\varepsilon) y^2)$ for   $\varepsilon > 0$, $x, y \in \RR_+$  leads to
{ \begin{equation}\label{4.9}
\begin{aligned}
      \EE \left( \big| e^{k,\D} (t)\big|^2 \II_{\{\vo^k_\D > T\}} \right) 
      \le & \EE \int_0^{t} \Big[ \Big( 2\big\lan e^{k,\D}(s),f_k(x^k_s,\o(s))-f(X^{k,\D}_s,\o_s) \big\ran\\
      & + \big|g_k(x^k_s,\o(s))-g(X^{k,\D}_s,\o_s)\big|^2 \Big) \II_{\{\vo^k_\D > T \}} \Big]\d s \\
      \le & \EE \left( \Big( \mathcal{J}_1(t) + \mathcal{J}_2(t) + \mathcal{J}_3(t) \Big) \II_{\{\vo^k_\D > T \}} \right) ,
\end{aligned}
\end{equation}}
where
$$
\mathcal{J}_1(t) :=  \int_0^{t}2\big\lan e^{k,\D}(s), f_k(x^k_s,\o(s)) - f(Z^{k,\D}_s,\o(s))\big\ran  +(\td p-1)\big|g_k(x^k_s,\o(s))-g(Z^{k,\D}_s,\o(s))\big|^2\d s,
$$
$$
\mathcal{J}_2(t) := \int_0^{t} 2\big\lan e^{k,\Delta}(s), f(Z^{k,\D}_s,\o(s)) - f(X^{k,\D}_s,\o(s))\big\ran  + C_{\td p}\big|g(Z^{k,\D}_s,\o(s))-g(X^{k,\D}_s,\o(s))\big|^2 \d s,
$$
$$
\begin{aligned}
\mathcal{J}_3(t) := \int_0^{t} 2\big\lan e^{k,\D}(s),
f(X^{k,\D}_s,\o(s))-f(X^{k,\D}_s,\o_s)\big\ran + C_{\td p}\big|g(X^{k,\D}_s,\o(s))-g(X^{k,\D}_s,\o_s)\big|^2 \d s.\\
\end{aligned}
$$
{ It follows from Assumption \ref{a4.2} that for any $t \in [0, T] \subset [0,\vo^k_\Delta)$,}
\begin{equation}\label{J-1'}
\begin{aligned}
\mathcal{J}_1(t) & \le  \int_0^{t} \bigg( d_3   \int_{-\infty}^0 \big|\pi_k(x^k_s)(u)-Z^{k,\D}_s(u)\big|^2 \n_1(\d u) 
- U\left(x^k(s),Z^{k,\D}(s)\right)  \\
& + \int_{-\infty}^0 U \left(\pi_k(x^k_s)(u), Z^{k,\D}_s(u)\right)\n_2(\d u)  \bigg)  \d s.
\end{aligned}
\end{equation}
Moreover,
\begin{equation}\label{4.11}
\int_0^{t} \int_{-\8}^0 \big|\pi_k(x^k_s)(u)-Z^{k,\D}_s(u)\big|^2\n_1(\d u)\d s \le \mathcal{J}_{11}(t)+\mathcal{J}_{12}(t),
\end{equation}
where
$$
\begin{aligned}
&\mathcal{J}_{11}(t) := \int_0^{t} \int_{-k}^0 \big|x^k_s(u)-Z^{k,\D}_s(u)\big|^2\n_1(\d u)\d s,\\
&\mathcal{J}_{12}(t) := \n_1\big((-\infty,-k) \big)  \int_0^{t} \big|e^{k,\D}(s-k)\big|^2\d s.\\
\end{aligned}
$$
From the definition of $Z^{k,\D}(\cdot)$, it follows that
\begin{equation}\label{4-J11}
\begin{aligned}
 \mathcal{J}_{11}(t) =& \int_{-k}^0 \int_0^{t}  \big|e^{k,\D}(s+u)\big|^2 \d s \n_1(\d u) \\
\le & \int_{-k}^0 \int_u^0  \big|e^{k,\D}(s)\big|^2 \d s \n_1(\d u)+\int_{-k}^0 \int_0^{t}  \big|e^{k,\D}(s)\big|^2 \d s \n_1(\d u) \\
= & \n_1\big([-k,0]\big)\int_0^{t\we\vo^k_\D}\big|e^{k,\D}(s)\big|^2\d s.
\end{aligned}
\end{equation}
Similarly, 
\begin{equation}\label{4-J12}
\begin{aligned}
\mathcal{J}_{12}(t)
\le  \n_1\big((-\8,-k)\big)  \int_0^{t} \big|e^{k,\D}(s)\big|^2\d s.
\end{aligned}
\end{equation}
Inserting \eqref{4-J11} and \eqref{4-J12} into \eqref{4.11} gives
\begin{equation}\label{4.12}
\begin{aligned}
 \int_0^{t} \int_{-\8}^0 |\pi_k(x^k_s)(u)-Z^{k,\D}_s(u)|^2 \n_1(\d u)\d s
\le   \int_0^{t}|e^{k,\D}(s)|^2\d s.
\end{aligned}
\end{equation}
In a similar way, we have
\begin{equation}\label{4.13}
\begin{aligned}
 \int_0^{t}\int_{-\8}^0 U \left(\pi_k(x^k_s)(u),Z^{k,\D}_s(u)\right) \n_2(\d u)\d s 
\le   \int_0^{t} U \left(x^k(s),Z^{k,\D}(t)\right)\d s.
\end{aligned}
\end{equation}
Substituting \eqref{4.12} and \eqref{4.13} into \eqref{J-1'} gives
\begin{equation}\label{J-1}
\begin{aligned}
 \EE \left( \mathcal{J}_1(t) \II_{\{\vo^k_\D > T \}} \right)  \le d_3 \EE \int_0^{t} \left(  \big|e^{k,\D}(s)\big|^2 \II_{\{\vo^k_\D > T \}} \right) \d s.
\end{aligned}
\end{equation}
Furthermore, it follows from Young's inequality that
\begin{equation}\label{4.15}
\begin{aligned}
\EE \left( \mathcal{J}_2(t) \II_{\{\vo^k_\Delta > T \}} \right)\le& \EE \int_0^{t} \left(  \big|e^{k,\D}(s)\big|^2 \II_{\{\vo^k_\Delta > T \}} \right)\d s \\
 +  &  C_{\td p}\EE\int_0^{t} \bigg[\bigg(|f(Z^{k,\D}_s,\o(s))-f(X^{k,\D}_s,\o(s))|^2 \\
  +  &  |g(Z^{k,\D}_s,\o(s)) - g(X^{k,\D}_s,\o(s))|^2 \bigg) \II_{\{\vo^k_\Delta > T \}} \bigg] \d s.
\end{aligned}
\end{equation}
{  Following the estimation procedure employed in \cite[(3.79)]{LMS24}, and applying Assumption \ref{a4.3}, $p \ge (3v/2)+1$, together with Theorem \ref{th-BTN} and Lemmas \ref{l4.6}, \ref{l4.7}, we obtain, for any $s \in [0, t]$
\begin{align}\nn
& \EE \left( \big|f(Z^{k,\D}_s,\o(s)) -f(X^{k,\D}_s,\o(s)) \big|^2 \II_{\{ \vo^k_\Delta > T\}} \right)  \\ \nn
\le & C \bigg[\EE\Big(  \int_{-\8}^0 \big|Z^{k,\D}_s(u)-X^{k,\D}_s(u)\big|^{\frac{4p}{v+2}} \n_3(\d u) \Big)\bigg]^{\frac{v+2}{2p}} \\ \nn \nonumber
& \K \bigg[ 1 + \EE\left(  \int_{-\8}^0  \left( \left( \big|Z^{k,\D}_s(u)\big|^v +
\big|X^{k,\D}_s(u)\big|^v \right) \II_{\{ \vo^k_\Delta > T\}} \right) \n_4(\d u) \right)^{\frac{4p}{2p-v-2}}  \bigg]^{\frac{2p-v-2}{2p}} \\ \nn
\le & C_{T,\xi} \Delta.
\end{align}}
Similarly,
$$
\EE \left( \big|g(Z^{k,\D}_s,\o(s)) -g(X^{k,\D}_s,\o(s))\big|^2 \II_{\{ \vo^k_\Delta > T\}} \right) \le C_{T,\xi} \D.
$$
{  Inserting the above two inequalities into \eqref{4.15} yields
\begin{equation}\label{J-2}
\EE \left( \mathcal{J}_2(t) \II_{\{\vo^k_\Delta > T \}} \right) \le C_{T,\xi}\D + \EE \int_0^{t} \left(  \big|e^{k,\D}(s)\big|^2 \II_{\{\vo^k_\Delta > T \}} \right) \d s.
\end{equation}
Using Assumption \ref{a4.3}, Theorem \ref{th-BTN} and $p > (3v/2) + 1$, by the similar techniques in the proof of \cite[(5.15)]{LLMS23}, we obtain
\begin{equation}\label{J-3}
\begin{aligned}
\EE \left( \mathcal{J}_3(t) \II_{\{\vo^k_\Delta > T\}} \right) 
 \le & \EE \int_0^{t}  \left( |e^{k,\D}(s)|^2 \II_{\{\vo^k_\Delta > T \}} \right) \d s + \EE \int_0^t 
|f(X^{k,\D}_s,\o(s))-f(X^{k,\D}_s,\o_s)|^2 \d s \\
& + C_{\td p} \EE \int_0^t \big|g(X^{k,\D}_s,\o(s))-g(X^{k,\D}_s,\o_s)\big|^2 \d s\\
\le & C_{T,\xi}\D + \EE \int_0^{t} \left( \big|e^{k,\D}(s)\big|^2 \II_{\{\vo^k_\Delta > T\}} \right) \d s.
\end{aligned}
\end{equation}
Inserting \eqref{J-1}, \eqref{J-2} and \eqref{J-3} into \eqref{4.9} gives
\begin{equation}\label{4.23}
      \EE \left(  \big|e^{k,\D}(t) \big|^2 \II_{\{\vo^k_\Delta > T \}} \right) \le  C_{T,\xi}\D + C \int_0^{t} \EE \left( \big|e^{k,\D}(s)\big|^2 \II_{\{\vo^k_\Delta > T\}} \right) \d t.
\end{equation}
An application of the Gronwall inequality leads to
$$
   \EE \left(  \big|e^{k,\D}(t) \big|^2 \II_{\{\vo^k_\Delta > T \}} \right) \le C_{T,\xi}\D.
$$}
Inserting this inequality into \eqref{4-rate} gives 
$$
\EE \big| x^k(t) - X^{k,\D}(t) \big|^2 
\le C_{T,\xi} \Delta.
$$
Then the desired assertion follows from the fact that $C_{T,\xi}$ is independent of $k$ and $t$. The proof is complete.
\end{proof}
$\hfill\square$

We now establish the strong convergence rate of the TEM numerical solution to the exact solution of the HSFDEswID \eqref{ISFDE} by combining Lemma \ref{th4.5} and Theorem \ref{l4.9}.
\begin{theorem}\label{th4.11}
Let the assumptions in Theorem \ref{l4.9} hold. Then,
$$
      \sup_{0 \le t \le T} \EE \big|x(t)-X^{k,\D}(t)\big|^2 \le  C_T \left( e^{-\alpha k} + \D\right), \quad \forall T > 0,~ k \ge T, \hbox{ and } \Delta \in (0, 1],
$$
where $\alpha$ is defined in Lemma \ref{th4.5}.
\end{theorem}

Especially, in Theorem \ref{th4.11}, for  $\D \in (0, 1]$, let
 $k_{\D} = \left\lfloor  \frac{- \ln(\D/K)}{\alpha} \right\rfloor.
$
Write $X^{k_{\D},\D}(T)=X^\D(T) $ for short. Then, by virtue of Theorem \ref{th4.11}, it follows that the TEM numerical solution converges to the exact one of \eqref{ISFDE} with  $1/2$ convergence order.
\begin{coro}\label{rmk4.12}
Let the assumptions in Theorem \ref{l4.9} hold. Then,
$$
\EE \big|x(T)-X^\D(T)\big|^2 \le C_{T,\xi} \D, \quad \forall T > 0.
$$
\end{coro}

\section{Exponential stability}\label{S5}

Exponential stability is one of important dynamical properties in the infinite-time horizon, which attracts  much attention in many fields such as automatic control and others. This section first establishes the exponential stability of the exact solution for \eqref{ISFDE}, including in both moment and almost sure means. Then, we have to face two questions:
\begin{itemize}
  \item Can the truncated equation  \eqref{TSFDE} keep the stability?
  \item Which numerical method can preserve the underlying stability of \eqref{ISFDE}?
\end{itemize}
 This section attempts to answer the two questions above.

Now we begin with introducing the structure of the probability space. Owing to the independence of the Brownian motion $\{B(t)\}_{t\ge 0}$ and the Markov chain $\{\o(t)\}_{t\ge 0}$,  we introduce a product probability space.
Let $(\W_1, \CF^1, \{\CF^1_t\}_{t \ge 0}, \PP_1)$ be a filtered probability space on which $\{B(t)\}_{t\ge0}$ is a $d$-dimensional Brownian motion, and $(\W_2, \CF^2, \{\CF^2_t\}_{t\ge 0}, \PP_2)$ be another filtered probability space on which $\{\o(t)\}_{t\ge 0}$ is a continuous-time Markov chain.
Let
$$
(\W, \CF, \{\CF_t\}_{t \ge 0}, \PP)=(\W_1\K\W_2, \CF^1\K\CF^2, \{\CF^1_t\K\CF^2_t\}_{t \ge 0}, \PP_1\K\PP_2).
$$
We use $\w = (\w_1, \w_2)$ to denote a generic element of $\W$. Let $Z$ be a random variable defined on the product space $\Omega= \Omega_1 \times \Omega_2$. For convenience, for any fixed $\omega_1 \in \Omega_1$, define random variable $Z^{\omega_1}$ on $\Omega_2$ by
$$
Z^{\omega_1}(\omega_2) = Z(\omega_1, \omega_2),
$$ 
and similarly, for any fixed $\omega_2 \in \Omega_2$, define random variable $Z^{\omega_2}$ on $\Omega_1$ by
$$
Z^{\omega_2}(\omega_1) = Z(\omega_1, \omega_2).
$$  In the sequel, we use the mutually independent processes $\w \mapsto B(t, \w_1)$ and $\w \mapsto \o(t, \w_2)$.
That is, for almost $\w = (\w_1, \w_2)$, $B(t)$ depends only on $\w_1$ and $\o(t)$ depends only on $\w_2$. 
Let $\EE_1$ and $\EE_2$ denote the expectations with respect to $\PP_1$ and $\PP_2$.
In this section, we  always assume $\o(t)$ is irreducible (i.e., the linear equation $\varpi Q =0$ and $\sum_{i=1}^N \varpi_i = 1$ has a unique solution $\varpi=(\varpi_1, \cdots, \varpi_N) \in \RR^N$ satisfying $\varpi_i > 0$ for any $i \in \SM$) and
\begin{equation}\label{0}
f(0,i) = 0, \quad g(0,i) = 0, \quad \forall i\in\SM.
\end{equation}
For convenience, for any vector $y = (y_1, \cdots, y_N)^{\rm T}$ and   constant $q > 0$, define
$$
\text{diag}(y) = \text{diag}(y_1, \cdots, y_N),\
Q_{q,y} = Q +  q \text{diag}(y),\
\eta_{q,y} =-\max_{u\in\text{spec}(Q_{q,y})}\text{Re}(u),
$$
where $\text{spec}(Q_{q,y})$ and $\text{Re}(u)$ denote the spectrum of $Q_{q,y}$ and the real part of $u$, respectively.

\begin{assp}\label{a5.2}
There exists a
  measure $\r_1 \in \CP_{2r}$ and constants $\a_i \in \RR$, $\be_i \in \RR_+$ such that
$$
2 \langle \p(0), f(\p,i) \rangle + |g(\p,i)|^2\\
\le \a_i |\p(0)|^2 + \be_i \int_{-\8}^0|\p(u)|^2\r_1(\d u), \quad \forall \p\in\CC_r,~~~~\forall~i \in \SM.
$$
\end{assp}

Under Assumptions \ref{Lip} and \ref{a5.2} one notices that SFDEs \eqref{ISFDE} and \eqref{TSFDE} has a unique global solution  $x(t)$ and $x^k(t)$ on $(-\8, \8)$, respectively. For convenience, define $$\g_i = \a_i + \r_1^{(-\hat\a)}\be_i, \quad \gamma'_i = \alpha_i + (1- { \epsilon}
+ \rho_1^{(-\hat a)})\beta_i, \quad a_i = \alpha_i + \beta_i, \quad \forall i \in \SM,$$ where $\epsilon = (\c \beta)^{-1} \hat \beta$ if $\c \beta > 0$, and $\epsilon = 0$ if $\c \beta=0$. Let $\g = (\g_1, \cdots, \g_N)$ and $\gamma' = (\g'_1, \cdots, \g'_N)$.

\subsection{Stability of the approximation solutions}

This subsections yields the moment exponential stability of the solutions to equations \eqref{ISFDE} and \eqref{TSFDE}.
\begin{theorem}\label{sta1} 
Let Assumptions \ref{Lip} and \ref{a5.2} hold.
\begin{itemize}
\item[$(\romannumeral1)$] Assume that $\hat \alpha < 0$, $\eta_{1,\g} > 0$  and  there exists a constant $\kappa>0$ such that $\r_1 \in \CP_{(-\hat\a + \kappa)\ve (2r) }$. Then the solutions of \eqref{ISFDE} and \eqref{TSFDE} have the property
$$
\EE|x(t)|^2 \vee \left(\sup_{k \ge 1} \EE |x^k(t)|^2\right) \le C \|\xi\|_r^2 e^{-\eta_{1,\g} t}, \quad \forall t \ge 0.
$$
\item[$(\romannumeral2)$] Assume that $\hat a < 0$, $\eta_{1,\g'} > 0$ and there exists a constant $\kappa' > 0$ such that $\r_1 \in \CP_{(-\hat a + \kappa') \ve (2r) }$. Then the solutions of \eqref{ISFDE} and \eqref{TSFDE} have the property
$$
\EE|x(t)|^2 \vee \left(\sup_{k \ge 1} \EE |x^k(t)|^2\right) \le C \|\xi\|_r^2 e^{-\eta_{1,\g'} t}, \quad \forall t \ge 0.
$$
\end{itemize}
Here $C $ is a positive constant independent of $k$ and $t$.
\end{theorem}

\begin{proof}
$(\romannumeral1)$ To emphasize that the Lyapunov exponent is independent of $k$, we provide only the proof of the exponential stability of $x^k(t)$. The exponential stability of $x(t)$ follows using the same argument.
For any $k \ge 1$ and $\w_2 \in \W_2$, consider the following equation
$$
\d x^{k,\w_2}(t)=f_k(x^{k,\w_2}_t,\o^{\w_2}(t))\d t + g_k(x^{k,\w_2}_t,\o^{\w_2}(t))\d B(t),~ t>0,
$$
with initial data $x^{k,\w_2}_0=\xi$ and $\o^{\w_2}(0)=i_0$.
Applying It$\hat{\text{o}}$'s formula, Assumption \ref{a5.2} and the definition of $\pi_k$, we obtain that
\begin{equation}\label{5.12}
\begin{aligned}
e^{-\int_0^t \a_{\o^{\w_2}(s)}\d s} \EE_1 |x^{k,\w_2}(t)|^2
\le & \|\xi\|_r^2 + \EE_1 \int_0^t e^{-\int_0^s \alpha_{\theta^{\omega_2(v)}}\d v} \beta_{\theta^{\omega_2}(s)} \int_{-\infty}^0 |\pi_k(x^{k,\omega_2}_s)|^2 \rho_1(\d u) \d s \\
\le &  \|\xi\|_r^2  + \EE_1 I_1^{k,\w_2}(t)  + \r_1\big((-\infty,-k]\big) \EE_1 I_2^{k,\w_2}(t) ,
\end{aligned}
\end{equation}
where
$$
I_1^{k,\w_2}(t) =\int_0^t e^{-\int_0^s \a_{\o^{\w_2}(v)}\d v}
\be_{\o^{\w_2}(s)} \int_{-k}^0 |x^{k,\w_2}(s+u)|^2 \r_1(\d u)\d s.
$$
$$
I_2^{k,\w_2}(t)=  \int_0^t e^{-\int_0^s \a_{\o^{\w_2}(v)}\d v}
\be_{\o^{\w_2}(s)} |x^{k,\w_2}(s-k)|^2 \d s,
$$
Applying the Fubini theorem, then
making a change of variable from $s-k$ to $s$ and using $\alpha_i \ge \hat{\alpha}$ for any $i \in \SM$, we derive that
$$
\begin{aligned}
I^{k,\omega_2}_1(t) & =  \int_0^t e^{-\int_0^s \a_{\o^{\w_2}(v)}\d v}\be_{\o^{\w_2}(s)} \int_{-k}^0 |x^{k,\w_2}(s+u)|^2 \r_1(\d u)\d s \\
& =  \int_{-k}^0 \int_u^{t+u} e^{-\int_0^{s-u} \a_{\o^{\w_2}(v)}\d v}\be_{\o^{\w_2}(s-u)}  |x^{k,\w_2}(s)|^2  \d s  \r_1(\d u) \\
& \le  \c{\be} \|\xi\|_r^2 \int_{-k}^0 e^{\hat{\a}u} \int_u^0 e^{-(2r + \hat{\a})s} \d s \r_1(\d u) \\
&+   \int_0^t \left( e^{-\int_0^{s} \a_{\o^{\w_2}(v)} \d v} |x^{k,\w_2}(s)|^2 \right) \left( \int_{-k}^0 \be_{\o^{\w_2}(s-u)} e^{ \hat{\a } u  } \r_1(\d u) \right)\d s.
\end{aligned}
$$
Similarly,
$$
I^{k,\omega_2}_2(t) 
\le   \c{\be} \|\xi\|_r^2 e^{-\hat{\a} k} \int_{-k}^0 e^{-(2r + \hat{\a})s} \d s  
+  \int_0^{t} \Big( e^{-\int_0^{s} \a_{\o^{\w_2}(v)}\d v} |x^{k,\w_2}(s)|^2 \Big)   \Big( \be_{\o^{\w_2}(s+k)} e^{- \hat{\a} k  } \Big)\d s.
$$
Inserting the above inequalities into \eqref{5.12}, we derive that
\begin{equation*}
\begin{aligned}
& e^{-\int_0^t \a_{\o^{\w_2}(s)}\d s} \EE_1 |x^{k,\w_2}(t)|^2  
\le   \big( 1 + \c{\be} I_3^{k}\big)  \|\xi\|_r^2
+   \int_0^t   e^{-\int_0^s \a_{\o^{\w_2}(v)} \d v} \EE_1|x^{k,\w_2}(s)|^2  I_4^{k,\w_2}(s)  \d s,
\end{aligned}
\end{equation*}
where
$$
I_3^{k} =  \r_1\big((-\infty,-k]\big)e^{ -\hat{\a} k } \int_{-k}^0 e^{ -(2r + \hat{\a})s } \d s  +  \int_{-k}^0 e^{ \hat{\a} u } \int_u^0 e^{ -(2r + \hat{\a})s } \d s \r_1(\d u),
$$
$$
I_4^{k,\w_2}(s) = \r_1((-\infty,-k]) \be_{ \o^{\w_2}(s+k) }  e^{ -  \hat{\a} k } +  \int_{-k}^0  \be_{ \o^{\w_2}(s-u) }  e^{ \hat{\a} u}  \r_1 (\d u).$$
Then, Gronwall's inequality (\cite[Lemma 8., p.52]{K93}) is applicable so that
\begin{equation*}
\EE_1 |x^{k,\w_2}(t)|^2
\le    \big( 1 + \c{\be} I_3^{k}\big) \|\xi\|_r^2   e^{\int_0^t \left( \a_{\o^{\w_2}(s)} +  I_4^{k,\w_2}(s) \right) \d s }.
\end{equation*}
Taking expectation with respect to $\PP_2$ yields
\begin{equation}\label{5.13}
\EE |x^k(t)|^2
\le  \big( 1 + \c{\be} I_3^{k}\big) \|\xi\|_r^2  \EE_2 e^{\int_0^t \left( \a_{\o(s)} + I_4^{k}(s) \right) \d s },
\end{equation}
where $I_4^{k}(s)$ is defined analogously to $I_4^{k,\omega_2}(s)$, with $\theta^{\omega_2}(\cdot)$ replaced by $\theta(\cdot)$.
We will now proceed to estimate $\int_0^t I_4^{k}(s)  \d s$. By the Fubini theorem, we obtain
\begin{equation*}
\begin{aligned}
&   \int_0^t \int_{-k}^0    \be_{ \o(s-u) } e^{ \hat\a u } \r_1 (\d u) \d s
=  \int_{-k}^0 \int_{-u}^{t-u} e^{ \hat\a u } \be_{\o(s)} \d s \r_1(\d u) \\
\le & \int_{-k}^0 \int_t^{t-u} e^{ \hat\a u } \be_{\o(s)} \d s \r_1(\d u) + \int_{-k}^0 \int_0^{t} e^{ \hat\a u } \be_{\o(s)} \d s \r_1(\d u) \\
\le & \c\be  \int_{-k}^0  (-u) e^{ \hat\a u } \r_1(\d u)  + \int_{-k}^0  e^{ \hat\a u }  \r_1(\d u)  \int_0^{t} \be_{\o(s)}   \d s.
\end{aligned}
\end{equation*}
Similarly,
$$
\int_0^t \be_{ \o(s+k) } \d s  \le \c\beta k + \int_0^t \beta_{\theta(s)} \d s.
$$
Hence, 	\begin{equation}\label{5-5.7}
\begin{aligned}
\int_0^t I_4^k(s) \d s  & \le  \c\beta \left(  \rho_1\big((-\infty,-k]\big) k e^{-\hat\a k} + \int_{-k}^0 (-u)e^{\hat\a u} \rho_1(\d u) \right) \\
& +  \left(  \rho_1\big((-\infty,-k]\big)  e^{-\hat\a k} + \int_{-k}^0 e^{\hat\a u} \rho_1(\d u) \right) \int_0^t \beta_{\theta(s)} \d s.
\end{aligned}
\end{equation}
One also notes from $\r_1 \in \CP_{(-\hat\a+\kappa)\vee(2r)}$ that
\begin{equation}\label{rho_1}
\begin{aligned}
&  \rho_1\big((-\infty,-k]\big) k e^{-\hat\a k} + \int_{-k}^0 (-u)e^{\hat\a u} \rho_1(\d u) \\
\le &  \sup_{u \le 0} \left(-u e^{\kappa u} \right)   \left( \rho_1\big((-\infty,-k]\big) e^{-(\hat\alpha-\kappa)k} +  \int_{-k}^0  e^{(\hat\a-\kappa) u}   \rho_1(\d u)  \right) 
\le  \frac{ \rho_1^{(-\hat\a + \kappa)}}{e \kappa}.
\end{aligned}
\end{equation}
Similarly, 
$$
\rho_1\big((-\infty,-k]\big)  e^{-\hat\a k} + \int_{-k}^0 e^{-\hat\a u} \rho_1(\d u) \le \rho_1^{(-\hat\a)}.
$$
Substituting the above inequalities into \eqref{5-5.7} gives
\begin{equation}\label{5-5.7'}
\int_0^t I_4^k(s) \d s   \le  c_1 +  \rho_1^{(-\hat\a)} \int_0^t \beta_{\theta(s)} \d s,
\end{equation}
where $c_1 = \c \beta \rho_1^{(-\hat\a + \kappa)}/e \kappa$.
It can be shown in the same way as \eqref{rho_1} that
\begin{align}\nn
I^k_3 \le \frac{\c\be \|\xi\|_r^2 \r_1^{(-\hat\a+\kappa)} }{e\kappa}, \hbox{ if } 2r+\hat\a=0 \hbox{ and } I^k_3 \le \frac{\c\be \|\xi\|_r^2 \r_1^{\left((2r) \vee (-\hat\alpha) \right)} }{|2r+\hat\a|}, \hbox{ if } 2r+\hat\a \neq 0,
\end{align}
which means
\begin{align}\label{CI2}
I_3 := \sup_{k \ge 1}I_2^k   < +\infty.
\end{align}
Inserting \eqref{5-5.7'} and \eqref{CI2} into \eqref{5.13} implies
\begin{equation*}
\sup_{k \ge 1}\EE |x^k(t)|^2 \le c_2 \|\xi\|_r^2 \EE_2 e^{\int_0^t (\a_{\o(s)} + \r_1^{(-\hat\a)}\be_{\o(s)}) \d s},
\end{equation*}
here $c_2 =   \left( 1 + \c{\be} I_3 \right)e^{c_1}$.
By using \cite[Proposition 4.1]{BG10}, we obtain
\begin{equation}\label{sta-m}
\sup_{k \ge 1}  \EE |x^k(t)|^2 \le c_2 \|\xi\|_r^2 e^{-\eta_{1,\g} t}.
\end{equation}

$(\romannumeral2)$ The proof is similar to $(\romannumeral1)$; it only needs to apply the It$\hat{\rm o}$'s formula for $e^{-\int_0^t a_{\theta(s)} \d s  }|x^k(t)|^2$ instead of $e^{-\int_0^t \alpha_{\theta(s)} \d s  }|x^k(t)|^2$, and thus we omit it. The proof is complete.
\end{proof}$\hfill\square$ 

\begin{rmk}
In this remark, we provide some sufficient conditions which ensure that the assumptions ``$\hat\alpha < 0$,
$\eta_{1,\gamma} > 0$" and ``$\hat a < 0$,
$\eta_{1,\gamma'} > 0$" in Theorem \ref{sta1} hold. Since the state space $\SM$ is finite and $Q$-matrix is irreducible, $\theta(t)$ has a unique stationary distribution denoted by $\varpi = (\varpi_1, \cdots, \varpi_N)^{\rm T} \in \RR^{N}$. According to \cite[Theorem 1.5]{BG10}, if
\begin{equation}\label{5-5.10}
\sum_{i=1}^N \varpi_i \gamma_i < 0,
\end{equation}
there is a constant $\sigma_\gamma > 0$ such that $\eta_{q,\gamma}>0$ for $q \in (0, \sigma_\gamma)$ and $\eta_{q,\gamma} < 0$ for $q \in (\sigma_\gamma, +\infty).$ Hence, $\eta_{1,\gamma} > 0$ when $\sigma_\gamma > 1$. Moreover, \eqref{5-5.10} implies $\hat\alpha < 0$, since $\gamma_i \ge \alpha_i$ for all $i \in \SM$. In conclusion, the assumption  ``$\hat\alpha < 0$,
$\eta_{1,\gamma} > 0$" holds if \eqref{5-5.10} and $\sigma_\gamma > 1$ hold. Similarly, ``$\hat a < 0$,
$\eta_{1,\gamma'} > 0$" can be deduced by
$$
\sum_{i=i}^N \pi_i \gamma'_i < 0  \hbox{ and } \sigma_{\gamma'} > 1.
$$
\end{rmk}

In general, the $2$th moment exponential stability of the solutions does not imply that it is the almost surely  exponentially  stable. 
We impose an additional assumption in order for 
the almost sure exponential stability.

\begin{assp}\label{a5.4}
There exist a positive constant $K$ and a probability measure $\rho_2 \in \CP_{2r}$ such that
$$
|g(\p,i)|\\
\le K   \int_{-\infty}^0|\p(u)| \r_2(\d u), \quad \forall \p\in\CC_r,~ i \in \SM.
$$
\end{assp}

\begin{theorem}\label{sta2}
Let Assumptions \ref{Lip}, \ref{a5.2} and \ref{a5.4} hold.
\begin{itemize}
\item[$(\romannumeral1)$] 
Let the assumptions in Theorem \ref{sta1} $(\romannumeral1)$ hold. Then for any $\lambda \in (0, (2r)\wedge\eta_{1,\gamma})$, the solutions of \eqref{ISFDE} and \eqref{TSFDE} satisfy
$$
\limsup_{t \rightarrow \8} \frac{\log(|x(t)|)}{t} \vee \left[\sup_{k \ge 1} \left( \limsup_{t \rightarrow \8} \frac{\log(|x^k(t)|)}{t} \right)\right] \le
-\frac{\lambda}{2} \quad a.s.
$$
\item[$(\romannumeral2)$] Let the assumptions in Theorem \ref{sta1} $(\romannumeral2)$ hold. Then for any $\lambda'\in (0, (2r)\wedge\eta_{1,\gamma'})$, the solutions of \eqref{ISFDE} and \eqref{TSFDE} satisfy
$$
\limsup_{t \rightarrow \8} \frac{\log(|x(t)|)}{t} \vee \left[\sup_{k \ge 1} \left( \limsup_{t \rightarrow \8} \frac{\log(|x^k(t)|)}{t} \right)\right] \le
-\frac{\lambda'}{2}, \quad a.s.
$$
\end{itemize}
\end{theorem}
\begin{proof}
$(\romannumeral1)$
{ For any $\lambda \in (0, (2r)\wedge\eta_{1,\gamma})$, by applying It$\hat{\rm o}$'s formula and the proof techniques similar to those used in \cite[Theorem 3.1]{LLMS23} and \cite[$(4.14)$]{WYM17}, we obtain}
$$
 \EE \Big( \sup_{0\le s\le t}e^{\lbd s}|x^k(s)|^2 \Big)
\le  C \|\xi\|_r^2 + C  \int_0^t e^{\lbd u}  \EE  |x^k(u)|^2 \d u. 
$$
It follows from \eqref{sta-m} and $\lambda \in (0, (2r) \wedge \eta_{1,\gamma})$ that
$$
\EE \Big( \sup_{0\le s\le t}e^{\lbd s}|x^k(s)|^2 \Big)
\le  C  \|\xi\|_r^2. 
$$
Applying Fatou's lemma as $t \to \infty$, we obtain
\begin{equation}\label{5.8}
 \EE \Big( \sup_{0\le s < \8}e^{\lbd s}|x^k(s)|^2 \Big) \le C \|\xi\|_r^2.
\end{equation}
Hence,
$\sup_{0 \le s < \8}e^{\lbd s} |x(s)|^2 < \8 ~ a.s.$, which implies
$$
\limsup_{t \rightarrow \8}\frac{\log(|x^k(t)|)}{t} \le -\frac{\lbd}{2} \quad a.s.
$$
as desired.

$(\romannumeral2)$ The proof of conclusion $(\romannumeral2)$ is analogous to that of conclusion $(\romannumeral1)$ and is therefore omitted. This completes the proof.
\end{proof}
$\hfill\square$

\subsection{Stability of TEM solutions}
This subsection  aims to propose a more precise explicit numerical scheme such that its numerical solution realizes the underlying exponential stability in the infinite time horizon.

Under Assumption \ref{Lip}, by \eqref{0}, we may choose an increasing continuous function $\bar \G : \RR_+ \to \RR_+$ such that for any $\p\in\CC_r$ with $\|\p\|_r\le R$ and $i\in\SM$,
\begin{equation}\label{fg1}
\begin{aligned}
 & |f(\p,i)|^2  \le \bar \G^2(R) \int_{-\8}^0 |\p(u)|^2 \m_1(\d u), \\
 & |g(\p,i)|^2 \le \bar \G(R) \int_{-\8}^0|\p(u)|^2 \m_1(\d u).
\end{aligned}
\end{equation}
Assume that there is a positive integer $l$ such that $\D=1/l \in (0,1]$. Define a truncation mapping $\bar \L^\D: \RR^n \rightarrow \RR^n$ by
{ \begin{equation}\label{L1}
\bar \L^\D(x) = \Big( |x|\we {\bar \G}^{-1}(\D^{-\lbd}) \Big)\frac{x}{|x|}, \quad \forall x \in \RR^n,
\end{equation}}
where $\lbd \in (0,1/2)$ and $x/|x|=0$ if $x=0$. Then define the TEM scheme as follows
{ \begin{align}\label{TEMy}
\left\{\begin{aligned}
&U^{k,\D}(t_j)=\xi(t_j), \quad -kl \le j \le 0,\\
&V^{k,\D}({t_j})=\bar \L^{\D}(U^{k,\D}(t_j)),\ j \ge -kl,\\
&U^{k,\D}(t_{j+1})=V^{k,\D}(t_j)+f(V^{k,\D}_{t_j},\o_j)\D +
g(V^{k,\D}_{t_j},\o_j)\D B_j, \ j \ge 0.
\end{aligned}\right.
\end{align}}
where $V^{k,\D}_{t_j}$ is a $\mathcal{C}_r$-valued random variable which is defined by
\begin{equation}\label{LIV}
V^{k,\D}_{t_j}(u)=\\
\left\{\begin{aligned}
& \frac{t_{m+1}-u}{\D} V^{k,\D}(t_{j+m}) + \frac{u-t_m}{\D} V^{k,\D}(t_{j+m+1}),\\
 & ~~~~~~~~~~~~~~~~t_m \le u \le t_{m+1},~-k l \le m\le -1,\\
& V^{k,\D}(t_{j-k l}),~u\le -k  l,
\end{aligned}
\right.
\end{equation}
Furthermore, the continuous-time TEM numerical solution is given by
$$
V^{k,\D}(t) = V^{k,\D}(t_j), \quad \forall t \in [t_j, t_{j+1}).
$$
It follows from \eqref{fg1}, \eqref{L1} and \eqref{TEMy} that
{ \begin{equation}\label{linear2}
\begin{aligned}
 & |f(V^{k,\D}_{t_j},i)|^2 \le \D^{-2\lbd} \int_{-\8}^0 |V^{k,\D}_{t_j}(u)|^2 \m_1(\d u),\\
  & |g(V^{k,\D}_{t_j},i)|^2 \le \D^{-\lbd} \int_{-\8}^0 |V^{k,\D}_{t_j}(u)|^2 \m_1(\d u),\\
\end{aligned}
\end{equation}}
for any $j \ge 0$ and $i \in \SM$. Compared with \eqref{linear}, \eqref{linear2} has a more precise property, which allows  the numerical solution to realize the underlying stability. By the similar arguments  as those in Section \ref{S3}, one notices that the TEM scheme defined by \eqref{TEMy} and \eqref{LIV} has a storage cost of $O((kl+1)n)$.  

\begin{theorem}\label{sta3}
Let Assumptions \ref{Lip} and \ref{a5.2} hold. For some $\kappa>0$, if one of the following conditions hold,
\begin{itemize}
\item[$(\romannumeral1)$] $\hat\alpha < 0$, $\eta_{1,\gamma}>0$ and $\rho_1, \mu_1 \in \CP_{(-\hat\alpha + \kappa)\vee(2r)}$;
\item[$(\romannumeral2)$] $\hat a < 0$, $\eta_{1,\gamma'}>0$ and $\rho_1, \mu_1 \in \CP_{(-\hat a + \kappa)\vee(2r)}$,
\end{itemize}
then, there exist $\Delta^* \in (0, 1]$ and $\eta > 0$ such that
\begin{equation}\label{5.19}
\sup_{k \ge 1} \EE \big|V^{k,\D}(t_n)\big|^2 \le  C \|\xi\|_r^2 e^{-\eta t_n}, \quad \forall \D \in (0, \Delta^*].
\end{equation}
Moreover, for any $\varepsilon>0$, 
\begin{equation}\label{5.20}
\sup_{k \ge 1}  \left[ \limsup_{n \rightarrow \8} \frac{\log\left(\big|V^{k,\D}(t_n)\big|\right)}{t_n}  \right] 
\le - \frac{\eta-\varepsilon}{2}, \ a.s. \quad \forall \D \in (0, \Delta^*].
\end{equation}
\end{theorem}

\begin{proof}
$(\romannumeral1)$ 
We first prove \eqref{5.19}.
For any $\omega_2 \in \Omega_2$,  $\Delta \in (0,1]$ and $j \ge 0$,
let $\hat f^{k,\D,\omega_2}_j=f(V^{k,\D,\omega_2}_{t_j},\o^{\omega_2}_j)$,
$\hat g^{k,\D,\omega_2}_j=g(V^{k,\D,\omega_2}_{t_j},\o^{\omega_2}_j)$. In this proof, we use the notation $\alpha^{\omega_2}_j := \alpha_{\theta^{\omega_2}_j  }$, $\beta^{\omega_2}_j := \beta_{\theta^{\omega_2}_j  }$ for simplicity.
By \eqref{TEMy}, 
$$
\big|U^{k,\D,\omega_2}(t_{j+1})\big|^2 = \big|V^{k,\D,\omega_2}(t_j)\big|^2 + \hat \z^{k,\Delta,\omega_2}_j,
$$
where
$$
\begin{aligned}
\hat \z^{k,\Delta,\omega_2}_j &= \big|\hat f^{k,\D,\omega_2}_j\big|^2\D^2 + \big|\hat g^{k,\D,\omega_2}_j\D B_j\big|^2 + 2 \big\langle V^{k,\D,\omega_2}(t_j), \hat f^{k,\D,\omega_2}_j \big\rangle \D \\
 & + 2 \big\langle V^{k,\D,\omega_2}(t_j), \hat g^{k,\D,\omega_2}_j \D B_j \big\rangle  + 2 \big\langle \hat f^{k,\D,\omega_2}_j, \hat g^{k,\D,\omega_2}_j\D B_j \big\rangle \D.
\end{aligned}
$$
It follows from \eqref{linear2}, Assumption \ref{a5.2}, $\lbd \in (0,1/2)$ and the inequality $1 + x  \le e^{x}$ for any $x \in \RR$ that
\begin{equation}\label{5-dd}
\begin{aligned}
\EE_1 \big|V^{k,\D,\w_2}(t_{j+1})\big|^2 \le \EE_1 |U^{k,\Delta, \omega_2}(t_{j+1})|^2
\le &  \EE_1 \left(  \big( 1 +  \a^{\w_2}_j \D \big) \big|V^{k,\D,\w_2}(t_j)\big|^2 +   R^{k,\D,\omega_2}_{1j} +  R^{k,\D,\omega_2}_{2j} \right) \\
\le &  \EE_1 \left(  e^{ \a^{\w_2}_j \D} \big|V^{k,\D,\w_2}(t_j)\big|^2 +   R^{k,\D,\omega_2}_{1j} +  R^{k,\D,\omega_2}_{2j} \right) \\
\end{aligned}
\end{equation}
where 
$$
      R^{k,\D,\omega_2}_{1j} = \D  \be^{\w_2}_j \int_{-\8}^0 \big|V^{k,\D,\w_2}_{t_j}(u)\big|^2\r_1(\d u), ~
      R^{k,\D,\omega_2}_{2j} = L^2 \D^{2-2\lbd} \int_{-\8}^0 \big|V^{k,\D,\w_2}_{t_j}(u)\big|^2\m_1(\d u).
$$
First of all, we verify that for any $n \ge 1$,
\begin{equation}\label{5-sdd}
      e^{-\sum_{m=0}^{n-1} \alpha^{\omega_2}_{m} \Delta } \EE_1 \big| V^{k,\D,\w_2}(t_n) \big|^2 \le |\xi(0)|^2 + \sum_{j=0}^{n-1} e^{ -\sum_{m=0}^j \alpha^{\omega_2}_m \Delta } \EE_1 \left( R^{k,\D,\omega_2}_{1j} + R^{k,\D,\omega_2}_{2j} \right).
\end{equation}\label{tnt2}
{ For $n = 1$, by \eqref{5-dd} and 
\begin{equation}
|\bar\Lambda^{\Delta}(x)| \le |x|, \quad \forall x \in \RR^n,
\end{equation}
we have}
\begin{equation}\label{5-sdd1}
      e^{-\alpha^{\omega_2}_0 \Delta}  \EE_1 \big|V^{k,\D,\w_2}(t_1)\big|^2 \le    |\xi(0)|^2 + e^{-\alpha^{\omega_2}_0 \Delta} \EE_1 \left(  R^{k,\D,\omega_2}_{10} +  R^{k,\D,\omega_2}_{20} \right),
\end{equation}
which means \eqref{5-sdd} holds for $n =1$. Next, for any $j \ge 1$, multiplying both sides of \eqref{5-dd} by $e^{-\sum_{m=0}^{j} \a^{\w_2}_m\D}$ yields
$$
\begin{aligned}
& e^{-\sum_{m=0}^{j} \a^{\w_2}_m \D} \EE_1 \big|V^{k,\D,\w_2}(t_{j+1})\big|^2 \\
\le & e^{-\sum_{m=0}^{j-1} \a^{\w_2}_m \D} \EE_1 \big|V^{k,\D,\w_2}(t_j)\big|^2  
+   e^{-\sum_{m=0}^{j} \a^{\w_2}_m  \D} \EE_1 \left( R^{k,\Delta,\omega_2}_{1j} +  R^{k,\Delta,\omega_2}_{2j}  \right). \\
\end{aligned}
$$
For any positive integer $n \ge 2$, summing over $j$ from $1$ to $n-1$ yields
\begin{equation*}
\begin{aligned}
      & e^{-\sum_{m=0}^{n-1} \a^{\w_2}_m  \D} \EE_1 \big|V^{k,\D,\w_2}(t_{n}) \big|^2  \\
      \le &  e^{-\alpha^{\omega_2}_0 \Delta}\mathbb{E}_1 |V^{k,\D,\omega_2}(t_1)|^2  +  \sum_{j=1}^{n-1}  e^{-\sum_{m=0}^{j} \a^{\w_2}_m  \D} \EE_1 \left( R^{k,\Delta,\omega_2}_{1j} +  R^{k,\Delta,\omega_2}_{2j}  \right)
\end{aligned}
\end{equation*}
It then follows from \eqref{5-sdd1} that
\begin{equation*}
\begin{aligned}
      e^{-\sum_{m=0}^{n-1} \a^{\w_2}_m  \D} \EE_1 \big|V^{k,\D,\w_2}(t_{n}) \big|^2  
      \le   |\xi(0)|^2  +  \sum_{j=0}^{n-1}  e^{-\sum_{m=0}^{j} \a^{\w_2}_m  \D} \EE_1 \left( R^{k,\Delta,\omega_2}_{1j} +  R^{k,\Delta,\omega_2}_{2j}  \right).
\end{aligned}
\end{equation*}
Hence, \eqref{5-sdd} holds for any $n \ge 1$. By the definition of $V^{k,\D,\w_2}_{t_j}(\cdot)$, one has
\begin{equation} \label{R1j}
\begin{aligned}
       \sum_{j=0}^{n-1}  e^{-\sum_{m=0}^{j} \a^{\w_2}_m  \D}  R^{k,\Delta,\omega_2}_{1j} 
      = & \D \sum_{j=0}^{n-1} e^{-\sum_{m=0}^{j} \a^{\w_2}_m   \D} \be^{\w_2}_j \int_{-k}^0
 \big| V^{k,\D,\w_2}_{t_j}(u) \big|^2 \r_1(\d u)   \\
      + &\D \r_1 \big((-\8,-k]\big) \sum_{j=0}^{n-1} e^{-\sum_{m=0}^{j} \a^{\w_2}_m  \D} \be^{\w_2}_j  \big|V^{k,\D,\w_2}(t_j-k) \big|^2  \\
       := &  \mathcal{X}^{k,\D,\w_2} +  \mathcal{Y}^{k,\D,\w_2}.
\end{aligned}
\end{equation}
{ It follows from \eqref{LIV}, Jensen's inequality and \eqref{tnt2} that}
\begin{align} \nn \label{X'}
\mathcal{X}^{k,\D,\w_2} \le & \Delta \sum_{j=0}^{n-1} e^{ -\sum_{m=0}^j \alpha^{\w_2}_m \Delta } \be^{\w_2}_j \bigg[ \sum_{w=-kl}^{-1} \int_{t_w}^{t_{w+1}} \bigg( \frac{ t_{w+1}-u }{\D} |V^{k,\D,\w_2}(t_{j+w})|^2 \\ \nn
   & + \frac{u-t_w}{\D}|V^{k,\D,\w_2}(t_{j+w+1})|^2 \rho_1(\d u) \bigg) \bigg] \\ \nn
      \le& \Delta \sum_{w=-kl}^{-1} \int_{t_w}^{t_{w+1}} \bigg[ \frac{ t_{w+1}-u }{\D}\sum_{j=0}^{n-1} e^{ -\sum_{m=0}^j \alpha^{\w_2}_m \Delta } \be^{\w_2}_j  |V^{k,\D,\w_2}(t_{j+w})|^2 \\ \nn 
   & + \frac{u-t_w}{\D} \sum_{j=0}^{n-1} e^{ -\sum_{m=0}^j \alpha^{\w_2}_m \Delta } \be^{\w_2}_j |V^{k,\D,\w_2}(t_{j+w+1})|^2  \bigg]  \rho_1(\d u) \\ 
      \le& \Delta \sum_{w=-kl}^{-1} \int_{t_w}^{t_{w+1}} \bigg[ \frac{ t_{w+1}-u }{\D}\sum_{j=w}^{n-1+w} e^{ -\sum_{m=0}^{j-w} \alpha^{\w_2}_m \Delta } \be^{\w_2}_{j-w}  |V^{k,\D,\w_2}(t_j )|^2 \\ \nn
   & + \frac{u-t_w}{\D} \sum_{j=w+1}^{n+w} e^{ -\sum_{m=0}^{j-w-1} \alpha^{\w_2}_m \Delta } \be^{\w_2}_{j-w-1} |V^{k,\D,\w_2}(t_j)|^2  \bigg]  \rho_1(\d u) \\ \nn
         \le& \Delta \sum_{w=-kl}^{-1} \int_{t_w}^{t_{w+1}} \bigg[ \sum_{j=w}^{n-1}  \bigg( \frac{ t_{w+1}-u }{\D} e^{ -\sum_{m=0}^{j-w} \alpha^{\w_2}_m \Delta } \be^{\w_2}_{j-w}   \\ \nn
   & + \frac{u-t_w}{\D}  e^{ -\sum_{m=0}^{j-w-1} \alpha^{\w_2}_m \Delta } \be^{\w_2}_{j-w-1} \bigg) |V^{k,\D,\w_2}(t_j)|^2  \bigg]  \rho_1(\d u) \\ \nn
   :=& \mathcal{X}_1^{k,\D,\w_2} + \mathcal{X}_2^{k,\D,\w_2},
\end{align}
where
$$
\begin{aligned}
      \mathcal{X}^{k,\D,\w_2}_1 &=\Delta \sum_{w=-kl}^{-1} \int_{t_w}^{t_{w+1}} \bigg[ \sum_{j=w}^{-1}  \bigg( \frac{ t_{w+1}-u }{\D} e^{ -\sum_{m=0}^{j-w} \alpha^{\w_2}_m \Delta } \be^{\w_2}_{j-w}   \\
   & + \frac{u-t_w}{\D}  e^{ -\sum_{m=0}^{j-w-1} \alpha^{\w_2}_m \Delta } \be^{\w_2}_{j-w-1} \bigg) |\xi(t_j)|^2  \bigg]  \rho_1(\d u)
\end{aligned}
$$
and 
$$
\begin{aligned}
      \mathcal{X}^{k,\D,\w_2}_2 &=\Delta \sum_{w=-kl}^{-1} \int_{t_w}^{t_{w+1}} \bigg[ \sum_{j=0}^{n-1}  \bigg( \frac{ t_{w+1}-u }{\D} e^{ -\sum_{m=0}^{j-w} \alpha^{\w_2}_m \Delta } \be^{\w_2}_{j-w}   \\
   & + \frac{u-t_w}{\D}  e^{ -\sum_{m=0}^{j-w-1} \alpha^{\w_2}_m \Delta } \be^{\w_2}_{j-w-1} \bigg) |V^{k,\D,\w_2}(t_j)|^2  \bigg]  \rho_1(\d u).
\end{aligned}
$$
By simple calculations, we can conclude that
\begin{equation}\label{X1}
\begin{aligned} 
      \mathcal{X}^{k,\D,\w_2}_1 \le   &  \D  \c \be e^{-\hat \alpha \D} \sum_{w=-kl}^{-1} \int_{t_w}^{t_{w+1}} \bigg( \sum_{j=w}^{-1}    e^{ -\hat\a t_{j-w} }  |\xi(t_j)|^2  \bigg)  \rho_1(\d u) \\ 
\le & \D  \c \be e^{-\hat \alpha \D} \|\xi\|_r^2 \sum_{w=-kl}^{-1} \int_{t_w}^{t_{w+1}} \bigg( e^{\hat\a t_w} \sum_{j=w}^{-1}    e^{ -(2r+\hat\a) t_j }  \bigg)  \rho_1(\d u) \\
\end{aligned}
\end{equation}
and
\begin{equation}\label{X2}
\begin{aligned}
      \mathcal{X}^{k,\D,\w_2}_2 = &  \D  \sum_{w=-kl}^{-1} \int_{t_w}^{t_{w+1}} \bigg\{  \sum_{j=0}^{n-1}  \bigg[ \bigg( e^{-\sum_{m=0}^{j-1}\a^{\w_2}_m \D } \big| V^{k,\D,\w_2}(t_j) \big|^2 \bigg) \\
      & \K \bigg(  \frac{t_{w+1}-u}{\D}  e^{-\sum_{m=j}^{j-w} \a^{\w_2}_m  \D} \be^{\w_2}_{j-w}
       +   \frac{u-t_w}{\D}  e^{-\sum_{m=j}^{j-w-1} \a^{\w_2}_m \D} \be^{\w_2}_{j-w-1}  \bigg)  \bigg] \bigg\}\r_1(\d u) \\
       \le & \D e^{-\hat\a \D}  \sum_{w=-kl}^{-1} \int_{t_w}^{t_{w+1}}  \bigg\{ e^{\hat\a t_w}  \sum_{j=0}^{n-1}  \bigg[ \bigg( e^{-\sum_{m=0}^{j-1}\a^{\w_2}_m \D } \big| V^{k,\D,\w_2}(t_j) \big|^2 \bigg) \\
      & \K \bigg(  \frac{t_{w+1}-u}{\D}   \be^{\w_2}_{j-w} 
       +   \frac{u-t_w}{\D}   \be^{\w_2}_{j-w-1} \bigg)  \bigg] \bigg\}  \r_1(\d u), \\
\end{aligned}
\end{equation}
here we used the convention $\sum_{m=0}^{-1} \alpha^{\omega_2}_m = 0$.
Inserting \eqref{X1} and \eqref{X2} into \eqref{X'} leads to 
\begin{equation}\label{X}
\begin{aligned}
      \mathcal{X}^{k,\D,\w_2} & \le  \D  \c \be e^{-\hat \alpha \D} \|\xi\|_r^2 \sum_{w=-kl}^{-1} \int_{t_w}^{t_{w+1}} \bigg( e^{\hat\a t_w} \sum_{j=w}^{-1}    e^{ -(2r+\hat\a) t_j }  \bigg)  \rho_1(\d u)\\
      &+ \D e^{-\hat\a \D}  \sum_{w=-kl}^{-1} \int_{t_w}^{t_{w+1}}  \bigg\{ e^{\hat\a t_w}  \sum_{j=0}^{n-1}  \bigg[ \bigg( e^{-\sum_{m=0}^{j-1}\a^{\w_2}_m \D } \big| V^{k,\D,\w_2}(t_j) \big|^2 \bigg) \\
      & \K \bigg(  \frac{t_{w+1}-u}{\D}   \be^{\w_2}_{j-w} 
       +   \frac{u-t_w}{\D}   \be^{\w_2}_{j-w-1} \bigg)  \bigg] \bigg\}  \r_1(\d u), \\
\end{aligned}
\end{equation}
Similarly, we have
\begin{equation}\label{Y}
\begin{aligned}
      \mathcal{Y}^{k,\D,\w_2} 
      \le &  \D \c\be \rho_1\big((-\infty,-k]\big)  e^{-\hat\alpha (k+\D)}  \|\xi\|_r^2   \sum_{j=-k l}^{-1} e^{-(2r+\hat\a)t_j}  \\
       + & \D  \rho_1\big((-\infty,-k]\big)  e^{-\hat\alpha(k+\D)} \sum_{j=0}^{n-1} \bigg[ \bigg( e^{-\sum_{m=0}^{j-1} \a^{\w_2}_m \D}  \big| V^{k,\D,\w_2}(t_j) \big|^2 \bigg) \be^{\w_2}_{j+k l}  \bigg].
\end{aligned}
\end{equation}
Substituting \eqref{X}, \eqref{Y} into \eqref{R1j} results in
\begin{equation*}
\begin{aligned}
     &\sum_{j=0}^{n-1}  e^{-\sum_{m=0}^{j} \a^{\w_2}_m \D }  R^{k,\Delta,\omega_2}_{1j} 
\le    \c\be  e^{-\hat \a \D}  \|\xi\|_r^2   \mathcal{I}^{k,\D} +  \D \sum_{j=0}^{n-1} \left[ \left( e^{-\sum_{m=0}^{j-1} \a^{\w_2}_m \D } \big|V^{k, \D,\w_2}(t_j) \big|^2 \right) \Theta^{k, \D,\w_2}_j \right].
\end{aligned}
\end{equation*}
where
$$
      \mathcal{I}^{k,\D}= \D \r_1\big( (-\infty,-k] \big) e^{-\hat\a k} \sum_{j=-kl}^{-1} e^{-(2r + \hat\a )t_j} 
      +  \D \sum_{w=-kl}^{-1}  \int_{t_w}^{t_{w+1}} e^{\hat\a t_w} \left( \sum_{j=w}^{-1} e^{-(2r + \hat\a)t_j} \right) \r_1(\d u)
$$
and
$$
\begin{aligned}
      \Theta^{k,\D,\w_2}_j
      & =    e^{-\hat\a \D }  \sum_{w=-k l}^{-1} \int_{t_w}^{t_{w+1}} e^{\hat\a t_w } \left( \frac{t_{w+1}-u}{\D} \be^{\w_2}_{j-w} +  \frac{u-t_w}{\D} \be^{\w_2}_{j-w-1} \right) \r_1(\d u) \\
      & + \r_1 \big((-\8,-k]\big) e^{-\hat\a (k+\D)}  \be^{\w_2}_{j+k l}
\end{aligned}
$$
If $2r + \hat\a > 0$, it is easy to verify from $\r_1 \in \CP_{(-\hat\a + \kappa) \vee (2r)}$ that
\begin{equation*}
\begin{aligned}
   \mathcal{I}^{k,\D} &  \le \frac{ e^{ (2r+\hat\a)\D } }{ 2r+\hat\a } \bigg( \r_1\big((-\infty,-k]\big) e^{2r k } + \sum_{w=-kl}^{-1} \int_{t_w}^{t_{w+1}} e^{-2rt_w} \r_1(\d u)   \bigg)  \\
   & \le  \frac{ e^{ (2r+\hat\a)\D } }{ 2r+\hat\a } \bigg( \int_{-\infty}^{-k} e^{-2r u} \r_1(\d u)  + \sum_{w=-kl}^{-1} \int_{t_w}^{t_{w+1}} e^{-2rt_{w+1}} e^{2r\D} \r_1(\d u)   \bigg)  \\
   & \le \frac{ e^{ (4r+\hat\a)\D } }{ 2r+\hat\a }  \int_{-\infty}^{0} e^{-2r u} \r_1(\d u)  \le  \frac{ e^{ 4r+\hat\a } \r_1^{(2r)} }{ 2r+\hat\a }. \\
\end{aligned}
\end{equation*}
In the same way, we can show that
$$
\mathcal{I}^{k,\D} \le \frac{e^{-\hat\alpha} \r_1^{(-\hat\a + \kappa)}}{e \kappa}, \hbox{ if } 2r+\hat\a = 0 \hbox{ and } \mathcal{I}^{k,\D} \le \frac{e^{4r + \hat\a} \r_1^{(-\hat\a)}}{-(2r + \hat\a)}, \hbox{ if } 2r+\hat\a < 0.
$$
Hence,
$$
\mathcal{I}:= \sup_{0 < \D \le 1, k \ge 1}\mathcal{I}^{k,\D}   < +\infty.
$$
Therefore,
\begin{equation}\label{R1j'}
\begin{aligned}
      \sum_{j=0}^{n-1}  e^{-\sum_{m=0}^{j} \a^{\w_2}_m\D}  R^{k,\Delta,\omega_2}_{1j} \le   C \|\xi\|_r^2 +  \D \sum_{j=0}^{n-1} \left[ \left( e^{-\sum_{m=0}^{j-1} \a^{\w_2}_m \D } \big|V^{k,\D,\w_2}(t_j) \big|^2 \right) \Theta^{k,\D,\w_2}_j \right].
\end{aligned}
\end{equation}
Similarly, 
\begin{equation}\label{R2j}
      \sum_{j=0}^{n-1}  e^{-\sum_{m=0}^{j} \a^{\w_2}_m\D}  R^{k,\Delta,\omega_2}_{2j}  \le  C \|\xi\|_r^2  + \D^{2-2\lambda} \sum_{j=0}^{n-1} \Big[ \Big( e^{-\sum_{m=0}^{j-1} \a^{\w_2}_m \D } \big|V^{k,\D,\w_2}(t_j)\big|^2 \Big) \Upsilon^{k,\D} \Big],
\end{equation}
where
$$
\begin{aligned}
      \Upsilon^{k,\D} = e^{ -\hat\a \D} \sum_{w=-k l}^{-1} \int_{t_w}^{t_{w+1}}  e^{\hat\a t_w } \m_1(\d u)  +  \m_1\big((-\8,-k]\big) e^{-\hat\a(k+\D) } .
\end{aligned}
$$
Similar to the estimate of $\mathcal{I}^{k,\D}$,
$$
\sup_{0 < \D \le 1,k \ge 1}\Upsilon^{k,\D} \le e^{-2\hat\a} \m_1^{(-\hat\a)} =: \Upsilon < + \infty.
$$
Inserting the above inequality into \eqref{R2j} yields
\begin{equation}\label{R2j'}
      \sum_{j=0}^{n-1}  e^{-\sum_{m=0}^{j} \a^{\w_2}_m\D}  R^{k,\Delta,\omega_2}_{2j}  \le  C \|\xi\|_r^2  + \Upsilon \D^{2-2\lambda} \sum_{j=0}^{n-1} \left( e^{-\sum_{m=0}^{j-1} \a^{\w_2}_m \D } \big|V^{k,\D,\w_2}(t_j)\big|^2 \right),
\end{equation}
Substituting \eqref{R1j'} and \eqref{R2j'} into \eqref{5-sdd} implies
\begin{equation*}
\begin{aligned}
& e^{-\sum_{m=0}^{n-1} \alpha^{\omega_2}_{m} \Delta } \EE_1 \big| V^{k,\D,\omega_2}(t_n) \big|^2 \\
\le & C  \|\xi\|_r^2  + \D \sum_{j=0}^{n-1} \left[ \left( e^{-\sum_{m=0}^{j-1}\a^{\w_2}_m \D } \EE_1 \big| V^{k,\D,\w_2}(t_j) \big|^2 \right)  \left( \Upsilon \D^{1-2\lbd} +  \Theta^{k,\D,\w_2}_j \right)   \right].
\end{aligned}
\end{equation*}
By the well-known discrete Gronwall inequality we find that
\begin{equation}\label{5.33}
e^{-\sum_{m=0}^{n-1} \a^{\w_2}_m  \D } \EE_1 \big|V^{k,\D,\w_2}(t_n) \big|^2 \le C \|\xi\|_r^2 e^{\D \sum_{j=0}^{n-1} \left(\Upsilon \D^{1-2\lbd} + \Theta^{k,\D,\w_2}_j \right) }.
\end{equation}
It is easy to see that
\begin{equation}\label{Theta'}
      \D \sum_{j=0}^{n-1} \Theta^{k,\D,\w_2}_j = e^{-\hat\a \D }  \mathcal{R}^{k,\D,\w_2}+ \D  \r_1 \big((-\8,-k]\big) e^{-\hat\a(k+\D)} \sum_{j=0}^{n-1}  \be^{\omega_2}_{j+k l},
\end{equation}
where
$$
\begin{aligned}
\mathcal{R}^{k,\D,\w_2} := \D \sum_{w=-k l}^{-1} \int_{t_w}^{t_{w+1}} e^{\hat\a t_w } \left( \frac{t_{w+1}-u}{\D} \sum_{j=0}^{n-1} \be^{\w_2}_{j-w} +  \frac{u-t_w}{\D} \sum_{j=0}^{n-1} \be^{\w_2}_{j-w-1} \right) \r_1(\d u).
\end{aligned}
$$
By a substitution technique, 
\begin{equation}\label{R}
\begin{aligned}
      \mathcal{R}^{k,\D,\w_2}(n)
      \le &  \D \sum_{w=-k l}^{-1} \int_{t_w}^{t_{w+1}} e^{\hat\a t_w}  \left(  \sum_{j=-1-w}^{n-1-w}  \be^{\w_2}_{j} \right) \r_1(\d u) \\
      \le &  \sum_{w=-k l}^{-1} \int_{t_w}^{t_{w+1}} e^{\hat\a t_w}  \left(  \D  \sum_{j=0}^{n-1}  \be^{\w_2}_{j}  - \c\beta t_w \right) \r_1(\d u).
\end{aligned}
\end{equation}
Similarly,
\begin{equation}\label{5.36}
      \D  \sum_{j=0}^{n-1}  \be^{\w_2}_{j+k l} \le   \D \sum_{j=0}^{n-1} \beta^{\w_2}_j + \c\beta k.
\end{equation}
Inserting \eqref{R} and \eqref{5.36} into \eqref{Theta'} leads to
$$
\begin{aligned}
      \D \sum_{j=0}^{n-1} \Theta^{k,\D,\w_2}_j &  \le   e^{-\hat\a \D } \left( \sum_{w=-k l}^{-1} \int_{t_w}^{t_{w+1}} e^{\hat\a t_w}  \r_1(\d u) + \r_1\big((-\infty,-k]\big) e^{-\hat\a k} \right)  \left( \D \sum_{j=0}^{n-1} \beta^{\w_2}_j \right) \\
      &  +  \c\beta e^{-\hat\a \D } \left(  \sum_{w=-k l}^{-1} \int_{t_w}^{t_{w+1}}    \left( -    t_w e^{\hat\a t_w} \right) \r_1(\d u) + \r_1\big((-\infty,-k]\big) k e^{-\hat\a k} \right).
\end{aligned}
$$
By following a similar estimation of $\mathcal{I}^{k,\D}$ and noting that $\r_1 \in \CP_{(-\hat\a + \kappa) \vee (2r)}$, we can conclude that
$$
   \sum_{w=-k l}^{-1} \int_{t_w}^{t_{w+1}} e^{\hat\a t_w}  \r_1(\d u) + \r_1\big((-\infty,-k]\big) e^{-\hat\a k} \le e^{-\hat\a \D} \r_1^{(-\hat\a)}
$$
and
$$
\begin{aligned}
   & \sum_{w=-k l}^{-1} \int_{t_w}^{t_{w+1}}    \left( -    t_w e^{\hat\a t_w} \right) \r_1(\d u) + \r_1\big((-\infty,-k]\big) k e^{-\hat\a k}\\
   \le & \left( \sup_{u \le 0} (-u e^{\kappa u}) \right) \bigg( \sum_{w=-k l}^{-1} \int_{t_w}^{t_{w+1}}   e^{(\hat\a-\kappa) t_w} \r_1(\d u) +  \r_1\big((-\infty,-k]\big) k e^{-\hat\a k} \bigg) \le C.
\end{aligned}
$$
Hence,
$$
      \D \sum_{j=0}^{n-1} \Theta^{k,\D,\w_2}_j  \le \D e^{-2\hat\a \D} \r_1^{(-\hat\a)} \sum_{j=0}^{n-1} \beta^{\w_2}_j  + C \c{\beta} e^{-\hat\alpha \Delta},
$$
Inserting the above inequality into \eqref{5.33} results in
\begin{equation}
\begin{aligned}
      & \EE_1 \big|V^{k,\D,\w_2}(t_n) \big|^2 \\
      \le & C \|\xi\|_r^2 \exp\left\{\D \sum_{j=0}^{n-1} \left( \a^{\w_2}_j +  \Upsilon \D^{1-2\lbd} + \Theta^{k,\D,\w_2}_j \right) \right\} \\
      \le & C \|\xi\|_r^2 \exp \left\{ \left(\Upsilon \D^{1-2\lbd} + \c\beta \r_1^{(-\hat\a)}  \left( e^{-2\hat\a  \D} -1 \right) \right) t_n + \D \sum_{j=0}^{n-1} \left( \a^{\w_2}_j +  \r_1^{(-\hat\a)} \be^{\w_2}_j \right)  \right\}.
\end{aligned}     
\end{equation}
Taking expectations with respect to $\PP_2$ yields
\begin{equation}
\begin{aligned}
      \EE \big|V^{k,\D}(t_n) \big|^2 
      \le & C \|\xi\|_r^2  \exp \left\{ \left(\Upsilon \D^{1-2\lbd} + \c\beta \r_1^{(-\hat\a)}  \left( e^{-2\hat\a  \D} -1 \right) \right) t_n \right\}  \\
      \times & \EE_2\exp\left\{\D \sum_{j=0}^{n-1} \left( \a_{\o_j} +  \r_1^{(-\hat\a)} \be_{\o_j} \right)  \right\}.
\end{aligned}
\end{equation}
By virtue of $\eta_{1,\gamma} > 0$ and \cite[Lemma $4.1$]{BSY23}, we derive that there exists $\bar\eta>0$ such that
\begin{equation}\label{5.26}
      \EE \big|V^{k,\D}(t_n) \big|^2 \le  C \|\xi\|_r^2 \exp \left\{ \left(\Upsilon \D^{1-2\lbd} + \c\beta \r_1^{(-\hat\a)}  \left( e^{-2\hat\a  \D} -1 \right)  - \bar\eta \right) t_n \right\}.  
\end{equation}
Clearly,
$$
      \lim_{\Delta \rightarrow 0} \left(\Upsilon \D^{1-2\lbd} + \c\beta \r_1^{(-\hat\a)}  \left( e^{-2\hat\a  \D} -1 \right)  - \bar\eta \right) = - \bar\eta < 0.
$$
For any $\varepsilon \in (0, \bar\eta)$,
we can choose $\Delta^*\in (0,1]$ small enough such that for any $\D \in (0, \Delta^*]$,
$$
\Upsilon \D^{1-2\lbd} + \c\beta \r_1^{(-\hat\a)}  \left( e^{-2\hat\a  \D} -1 \right)  - \bar\eta  \le -\bar\eta+\e := - \eta.
$$
Substituting this into \eqref{5.26} implies \eqref{5.19}. Furthermore, by virtue of \eqref{5.19}, using the similar technique as in the proof of \cite[Theorem 4.2]{LMS24} implies \eqref{5.20}.

$(\romannumeral2)$ In this case, the desired assertion can be proved similarly to $(\romannumeral1)$. We only need to verify
$$
      e^{-\sum_{m=0}^{n-1} a^{\omega_2}_{m} \Delta } \EE_1 \big| V^{k,\D,\w_2}(t_n) \big|^2 \le |\xi(0)|^2 + \sum_{j=0}^{n-1} e^{ -\sum_{m=0}^j a^{\omega_2}_m \Delta } \EE_1 \left( R^{k,\D,\omega_2}_{1j} + R^{k,\D,\omega_2}_{2j} \right).
$$
instead of \eqref{5-sdd}, so we omit it. The proof is complete.
\end{proof}
$\hfill\square$

\section{Numerical experiments}\label{S6}

This section carries out several numerical experiments to support the  results  in Sections \ref{S4} and \ref{S5}. Let $B(t)$ be a scalar Brownian motion, and $\o(t)$ be a Markov chain on the state space $\mathbb{S}=\{1, 2\}$ with generator
$$
Q=\left(
\begin{array}{ccc}
-1  &  1 \\
2   &  -2 \\
\end{array}\right).
$$
Obviously, the Markov chain has the stationary distribution $\varpi = (2/3, 1/3)$. Consider a scalar stochastic functional volatility equation with infinite delay of the form
\begin{align}\label{eg-6.1}
\d x(t) = f(x_t, \o(t))\d t + g(x_t, \o(t))\d B(t),\ \ \ t > 0,
\end{align}
with initial data 
\begin{equation}\label{eg-initial}
\xi(u)=e^u \in \CC_1, \quad \theta(0)=1 \in \SM.
\end{equation}

\underline{Case 1.} ~~
The coefficients
$$
f(\f, 1) = -\f^3(0) + \int_{-\8}^0 \f(u)\m(\d u), \quad g(\f,1) = \f(0)
$$
$$
f(\f,2)=0.25\f(0)-\f^3(0)+ 0.25\int_{-\8}^0 \f(u)\m(\d u), \quad g(\f,2)= 0.5\f(0)
$$
for $\f\in\CC_1$, the probability measure $\m(\cdot)$ has the density $6e^{6u}$ on $(-\infty, 0]$ $($i.e., $\m(\d u)=6e^{6u}\d u$$)$.
Recalling the definition of truncation mapping $\pi_k$, we get the corresponding time-truncated HSFDE
\begin{align}\label{eg-6.1'}
\d x^k(t) = f_k(x^k_t, \o(t))\d t + g_k(x^k_t, \o(t))\d B(t),
\end{align}
with initial data \eqref{eg-initial}.
Here $f_k$ and $g_k$ have the forms
$$
f_k(\f,1)=-\f^3(0) + e^{-6k}\f(-k) + \int_{-k}^0\f(u)\mu(\d u), \quad g_k(\f,1)=\f(0),
$$
$$
f_k(\f,2)=0.25\f(0)-\f^3(0) + 0.25 e^{-6k}\f(-k) + 0.25 \int_{-k}^0\f(u)\mu(\d u), \quad g_k(\f,2)=0.5\f(0),
$$
for $\f\in\CC_1$. One notices that the drift and diffusion coefficients satisfy Assumptions \ref{Lip} and \ref{mon} with $p=3$. By virtue of Theorem \ref{th2.3} and Lemma \ref{lemma3.1}, SFDEs \eqref{eg-6.1} and \eqref{eg-6.1'} have unique global solutions $x(t)$ and $x^k(t)$ on $t \in (-\8, +\8)$, respectively.

Obviously, the initial data $\xi$ satisfies Assumption \ref{a4.1}. For any $\tilde p > 2$, we compute that
$$
\begin{aligned}
& 2\lan\p(0)-\f(0), f(\p,i)-f(\f,i)\ran + (\td p -1)|g(\p,i)-g(\f,i)|^2 \\
\le & (\tilde p+2) \Big( |\p(0)-\f(0)|^2 + \int_{-\8}^0 |\p(u)-\f(u)|^2 \mu(\d u) \Big),
\end{aligned}
$$
which implies Assumption \ref{a4.2} holds with $d_3 = \tilde p+2$, $U(\cdot,\cdot)\equiv 0$ and $\nu_1(\d u) = \mu(\d u)$.
Furthermore, we can verify that Assumption \ref{a4.3} holds with $d_4 = 1$, $v=2$, $\nu_3(\d u)=\mu(\d u)$, and $\n_4(\d u)=\nu_5(\d u)=\de_0(\d u)$. According to Remark \ref{rmk4.4}, we choose
$$
\G(R)=2(1+R^2), \quad \lbd=\frac{v}{2(p-1)}=\frac{1}{2},
$$
and $K_{10}=e^{35}$ such that $k_{\D} \ge 10$ for any $\Delta \in \{ 2^{-11}, 2^{-12}, 2^{-13}, 2^{-14}, 2^{-15}\}$. We denote the corresponding TEM numerical solution by $X^{\Delta}(t)$.
It follows from Corollary \ref{rmk4.12} that the TEM numerical solution $X^{\D}(T)$ converges to $x(T)$ with the rate of $1/2$.

To verify the efficiency of the TEM scheme, we carry out numerical experiments using MATLAB. Since \eqref{eg-6.1} cannot be solved explicitly, we regard the TEM numerical solution of \eqref{eg-6.1'} with $k=200$ and $\Delta = 2^{-16}$ as the exact solution $x(t)$ of \eqref{eg-6.1}.
Figure \ref{f1} depicts the root
mean square approximation error $\left(\EE |x(10)-X^{\D}(10)|^2\right)^{1/2}$ as a function of step size $\Delta \in \{ 2^{-11}, 2^{-12}, 2^{-13}, 2^{-14}, 2^{-15}\}$.
As illustrated in Figure \ref{f1}, the rate of convergence of the TEM numerical solution is close to $1/2$.
\begin{figure}
  \centering
  \includegraphics[width=8cm]{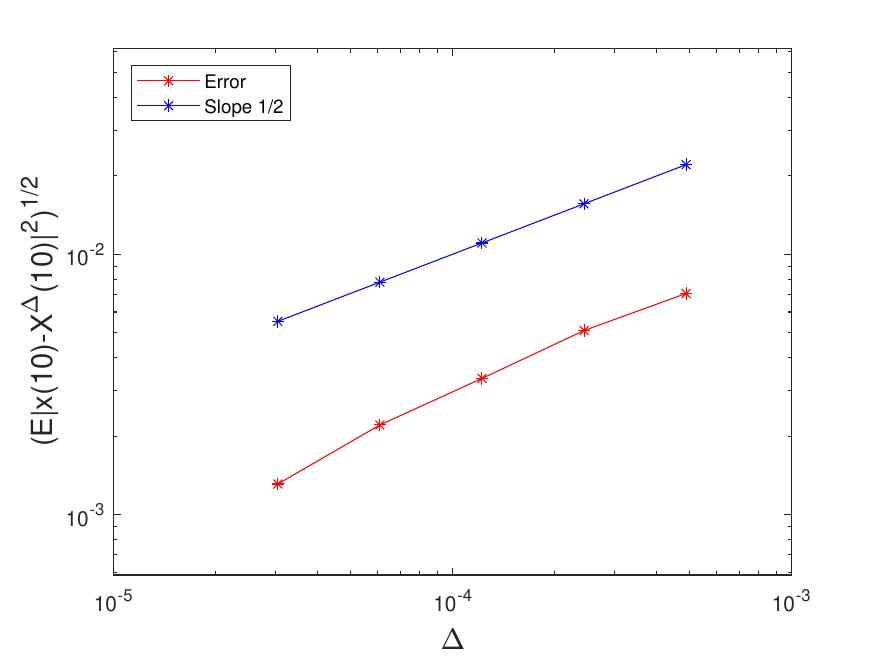}
  \caption{The root mean square approximation errors for 1000 sample points between the exact solution $x(10)$ of \eqref{eg-6.1} and the TEM numerical solution $X^\D(10)$ as a function of step size $\D \in \{2^{-11}, 2^{-12}, 2^{-13}, 2^{-14}, 2^{-15}\}$.}\label{f1}
\end{figure}

\underline{Case 2.} ~~ We take the drift coefficients of \eqref{eg-6.1} with the form
$$
f(\f, 1) = -4\f(0)-\f^3(0) + \int_{-\8}^0 \f(u)\m(\d u),
$$
$$
f(\f, 2) = 0.25\f(0)-\f^3(0)+ 0.25\int_{-\8}^0 \f(u)\m(\d u),
$$ and the diffusion coefficients are same as  in Case 1.
We choose
$$
\bar \G(R)=2(1+R^2), \quad \lbd=1/3 \in \left(0, 1/2 \right),~~\D= 2^{-10}.
$$
The corresponding TEM numerical solution is denoted by $V^{\Delta}(t)$.
It is easy to calculate that
$$
   \a_1=-6,\quad \a_2=1,\quad \be_1=1,\quad \be_2=1/4,\quad \g = (-19/5, 31/20).
$$
Hence,
$\eta_{1,\g} \approx 0.0308 > 0$.
By virtue of Theorems \ref{sta1}, \ref{sta2}, and \ref{sta3}, the corresponding exact solutions $x(t)$, $x^k(t)$, and the TEM numerical solution $V^{\D}(t)$ defined by \eqref{TEMy} are exponentially stable in the $2$th moment
and almost surely.
Figure \ref{f2} depicts the sample mean and $5$ sample paths of the TEM numerical solution $V^{\D}(t)$.
\begin{figure}
  \centering
  \includegraphics[width=14cm]{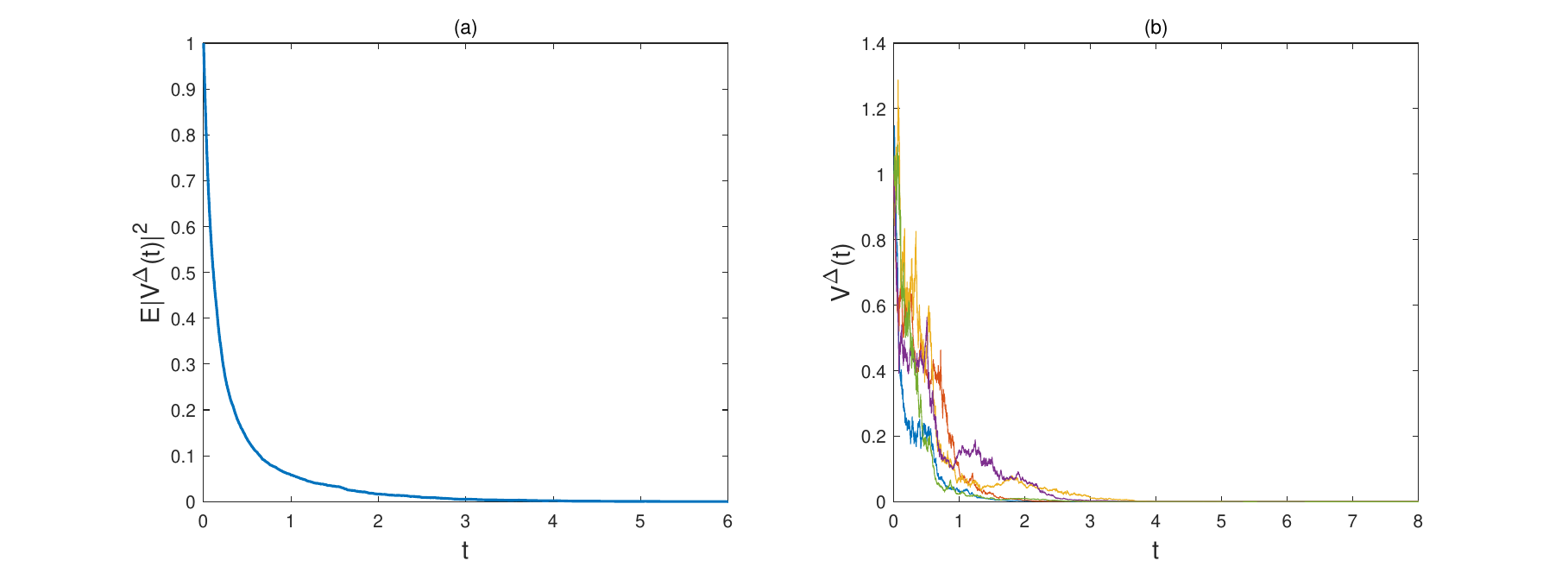}
  \caption{(a) The sample mean of $V^{\D}(t)$ for 2000 sample points with $\D=2^{-10}$. (b) $5$ sample paths of $V^{\D}(t)$ with $\D = 2^{-10}$.}\label{f2}
\end{figure}

\noindent
 \section{Conclusion Remark}
 This paper propose  an easily implementable explicit numerical method for   a class of super-linear HSFDEswID and establish its strong convergence along with $1/2$ order convergence rate by the time-truncation and space truncation techniques. Moreover, by the refined scheme the numerical solutions realize the exponential stability in the moment and almost sure means.
 Note that different functions $\G(\cdot)$ satisfying the inequality \eqref{G} and different parameter $\lambda$ in the permitted range correspond to different truncation mappings, which in turn leads to different numerical schemes. However, all these schemes are valid under the theories established in this paper, which highlights the flexibility of our theoretical framework.

\end{document}